\documentclass[a4paper, reqno]{amsart}
\usepackage{amssymb, amsmath, amsthm, chngcntr, enumerate, mathabx, mathrsfs, mathtools, pgf, tikz,esint}
\usepackage{graphicx}
\usepackage{dsfont}
\usetikzlibrary{arrows,calc,patterns}
\usepackage[english]{babel}
\usepackage[T1]{fontenc}

\numberwithin{equation}{section}
\newtheorem{theorem}{Theorem}[section]
\newtheorem*{theorem*}{Theorem}
\newtheorem{corollary}[theorem]{Corollary}
\newtheorem{lemma}[theorem]{Lemma}
\newtheorem{proposition}[theorem]{Proposition}

\theoremstyle{definition}

\theoremstyle{definition}
\newtheorem{remark}[theorem]{Remark}

\newtheorem{definition}[theorem]{Definition}
\newtheorem{definition lemma}[theorem]{Definition/Lemma}

\newcommand{\C}{\mathcal{C}}
\newcommand{\R}{\mathbb{R}}
\newcommand{\Z}{\mathbb{Z}}

\newcommand{\N}{\mathbb{N}}
\renewcommand{\SS}{\mathbb{S}}

\newcommand{\n}{\mathfrak{n}}
\newcommand{\m}{\mathfrak{m}}
\renewcommand{\P}{\mathbb{P}}

\newcommand{\ud}{\mathrm{d}}

\newcommand{\sizen}{\mathrm{Size}^{\mathfrak{n}}}

\newcommand{\energyn}{\mathrm{Energy}^{\mathfrak{n}}}
\renewcommand{\top}[1]{\mathrm{Top}(#1)}

\newcommand{\VarC}[1]{\mathrm{Var}^{#1}\mathscr{C}}

\newcommand{\ec}[1]{\mathrm{ec}(#1)}

\counterwithout{equation}{section}
\counterwithout{section}{section}

\author{Fr\'{e}d\'{e}ric Bernicot \and Marco Vitturi}

\address{CNRS - Universit\'e de Nantes \\ Laboratoire Jean
Leray \\ 2, rue de la Houssini\`ere
44322 Nantes cedex 3, France}
\email{frederic.bernicot@univ-nantes.fr \and marco.vitturi@univ-nantes.fr }

\title[Bilinear Rubio de Francia for rectangles]{A bilinear Rubio de Francia inequality for arbitrary rectangles}

\begin{document}

%
% abstract
%
\begin{abstract}
Let $\mathscr{R}$ be a collection of disjoint dyadic rectangles $R$, let $\pi_R$ denote the non-smooth bilinear projection onto $R$
\[ \pi_R (f,g)(x):=\iint \mathds{1}_{R}(\xi,\eta) \widehat{f}(\xi) \widehat{g}(\eta) e^{2\pi i (\xi + \eta) x} \ud \xi \ud\eta \]
and let $r>2$. We show that the bilinear Rubio de Francia operator associated to $\mathscr{R}$ given by
\[ f,g \mapsto \Big(\sum_{R\in\mathscr{R}} |\pi_{R} (f,g)|^r \Big)^{1/r} \]
is $L^p \times L^q \rightarrow L^s$ bounded whenever $1/p + 1/q = 1/s$, $r'<p,q<r$. This extends from squares to rectangles a previous result by the same authors in \cite{BernicotVitturi}, and as a corollary extends in the same way a previous result from \cite{BeneaBernicot} for smooth projections, albeit in a reduced range.
\end{abstract}

\maketitle
%
%%%%%%%%%%%%%%%%%%  
%
%   INTRODUCTION: DEFINITION OF THE BILINEAR OPERATOR
%
\section{Introduction}
In this paper we present an improvement over the results of Benea and the first author in \cite{BeneaBernicot} and the results of \cite{BernicotVitturi} by the same authors. The setup is as follows.\\ 
Let $\mathcal{D}$ denote the standard dyadic grid, that is the collection of intervals of the form $[2^k n, 2^k(n+1)]$ for $n,k \in \Z$. Let then $\mathscr{R} \subset \mathcal{D} \times \mathcal{D}$ be an arbitrary collection of disjoint rectangles $R = R_1 \times R_2$ with $R_1, R_2$ dyadic, and let $\pi_R$ denote the non-smooth bilinear frequency projection on $R$, that is 
\[ \pi_R (f,g)(x) := \iint \widehat{f}(\xi) \widehat{g}(\eta) \mathds{1}_{R}(\xi,\eta) e^{2\pi i (\xi + \eta) x} \ud \xi \ud \eta. \]
We will consider the bilinear operators associated to the collection $\mathscr{R}$ given by 
\[  T_{\mathscr{R}}^r(f,g)(x) := \Big(\sum_{R\in\mathscr{R}} |\pi_R(f,g)(x)|^r\Big)^{1/r}  \]
for $r>2$.
\par
Before stating our results we provide some context. Linear variants of the operator $T_{\mathscr{R}}^r$ are known in the literature as Rubio de Francia operators. Indeed, given a collection of arbitrary disjoint intervals $\mathcal{I} = \{I_n, n \in \Z\}$, Rubio de Francia proved in \cite{RubioDeFrancia} that the operator 
\[ \mathrm{RdF}^r_{\mathcal{I}}f(x) := \Big(\sum_{n} \Big|\int \widehat{f}(\xi) \mathds{1}_{I_n}(\xi) e^{2\pi i \xi x} \ud \xi\Big|^r \Big)^{1/r} \]
is $L^p \to L^p$ bounded for $r' < p < \infty$ when $r \geq 2$, this last condition being necessary (a consequence of Khinchin's inequality). Notice that when $r = 2$ then $\mathrm{RdF}^2_{\mathcal{I}}$ is also bounded at the endpoint $r' = 2$, this endpoint being just a consequence of Plancherel's identity, but when $r > 2$ the condition $r' < p$ is sharp (see \cite{TaoCowling}). Rubio de Francia's result is a generalization of the well known Littlewood-Paley inequalities ($I_n = [2^n,2^{n+1}]$), extending an earlier result of Carleson \cite{Carleson_LP} and (independently) Cordoba \cite{Cordoba} for the collection $\mathcal{I} = \{[n,n+1], n \in\Z\}$; as such, it can be thought of as a statement about orthogonality for arbitrary frequency intervals. The result has also been extended and reproved by different means in a number of papers, see \cite{Bourgain,Journe,Sjoelin,Soria,Lacey_RdF,BeneaMuscalu_preprint}.\\
Bilinear operators of square-function type associated to collections of subsets of the frequency plane $\widehat{\R^2}$ have previously been considered as well. Perhaps the first such operator to have appeared in the literature is the square function 
\[ f,g \mapsto \Big(\sum_{n \in \Z} \Big|\iint\widehat{f}(\xi) \widehat{g}(\eta) \chi(\xi-\eta-n) e^{2\pi i (\xi + \eta) x} \ud \xi \ud \eta\Big|^2 \Big)^{1/2}, \]
where $\chi$ is a smooth bump function supported in $[-1/2, 1/2]$. Here the collection of subsets of $\widehat{\R^2}$ is evidently given by the strips $S_n:= \{(\xi,\eta) \text{ s.t. } |\xi - \eta - n| < 1/2\}$ for $n \in \Z$, and the bilinear frequency projections are given by a smooth bilinear multiplier. This square function was introduced by Lacey in \cite{Lacey_bilinearSqF}, in which he proved it is $L^p \times L^q \mapsto L^2$ bounded for $2 \leq p,q \leq \infty$ satisfying $1/p + 1/q = 1/2$. This range was later extended in \cite{MohantyShrivastava},\cite{BernicotShrivastava} to show that the operator is $L^p \times L^q \mapsto L^s$ bounded for $1/p + 1/q = 1/s$, $1 \leq s\leq 2$ and $2 \leq p,q, \leq \infty$, this last condition being sharp.\\
The operator admits a non-smooth variant given by 
\[ f,g \mapsto \Big(\sum_{n \in \Z} \Big|\iint\widehat{f}(\xi) \widehat{g}(\eta) \mathds{1}_{[-1/2,1/2]}(\xi-\eta-n) e^{2\pi i (\xi + \eta) x} \ud \xi \ud \eta\Big|^2 \Big)^{1/2}, \]
where the collection of subsets is again given by the strips $S_n$ but now the bilinear frequency projections are given by a non-smooth multiplier - specifically, each projection is essentially a modulated Bilinear Hilbert Transform. The non-smoothness makes the operator inherently harder to bound. The first author proved in \cite{Bernicot} that the operator is $L^p \times L^q \to L^s$ bounded for $2 < p,q < \infty$, $1 < s < 2$ and $1/p + 1/q = 1/s$. It should be remarked that the geometric regularity of the strips (in particular the fact that they have all the same width and separation from their neighbours) is fundamental to the proof, the case of arbitrary disjoint strips being an interesting open problem.\\
Another example of a bilinear square-function associated to subsets of $\widehat{\R^2}$ can be found in \cite{BeneaMuscalu}, where the authors considered the operator
\[ f,g \mapsto \Big( \sum_{n\in\Z} \Big|\iint\limits_{a_n < \xi < \eta < a_{n+1}} \widehat{f}(\xi) \widehat{g}(\eta) e^{2\pi i (\xi + \eta) x} \ud \xi \ud \eta \Big|^r\Big)^{1/r} \]
for $r\geq 1$, where $(a_n)_{n\in \Z}$ is a strictly increasing subsequence of reals (a smoother version of this operator was also originally considered in \cite{Benea_thesis}). They can be thought of as bilinear Rubio de Francia operators for iterated Fourier integrals, and arise naturally in the study of the stability of solutions to AKNS systems of differential equations (see \cite{MuscaluSchlag}). Here we can see that the subsets of the frequency space $\widehat{\R^2}$ consist of disjoint right triangles whose hypotenuses are aligned along the $\xi = \eta$ diagonal, and the bilinear multipliers are non-smooth. In \cite{BeneaMuscalu} it is proven that when $r\geq 2$ the operator is $L^p \times L^q \to L^s$ bounded in the same range in which the Bilinear Hilbert Transform is bounded, and when $1 \leq r < 2$ the operator is still bounded but in a range depending on $r$.\\
The final example of a bilinear operator of the above kind we provide is the closest to the operator $T^r_{\mathscr{R}}$. Consider a collection $\Omega$ of disjoint \emph{squares} $\omega$ with sides parallel to the axes in $\widehat{\R^2}$. Let $\chi$ denote a smooth bump function supported in $[-1/2,1/2]\times [-1/2,1/2]$ and let $\chi_\omega$ be the rescaling of $\chi$ with support $\omega$, namely 
\[ \chi_\omega(\xi, \eta) = \chi \Big( \frac{\xi - c(\omega_1)}{|\omega_1|}, \frac{\eta - c(\omega_2)}{|\omega_2|}\Big), \]
$c(I)$ denoting the center of the interval $I$. For $r \geq 2$ fixed, one can associate to the collection $\Omega$ the bilinear Rubio de Francia operator
\[ f,g \mapsto S^r_{\Omega}(f,g)(x) := \Big(\sum_{\omega \in \Omega} \Big|\iint\widehat{f}(\xi) \widehat{g}(\eta) \chi_\omega (\xi,\eta) e^{2\pi i (\xi + \eta) x} \ud \xi \ud \eta\Big|^r \Big)^{1/r}. \]
Thus the sets are now squares with sides parallel to the axes and the frequency projections are smooth. In \cite{BeneaBernicot} by Benea and the first author the following is proven.
\begin{theorem}[\cite{BeneaBernicot}]\label{theorem_smooth_squares}
Let $r > 2$ and let $\Omega$ be a collection of disjoint squares. Let $p,q,s$ be such that 
\[ \frac{1}{p} + \frac{1}{q} = \frac{1}{s} \]
and
\[ r' < p,q < \infty, \qquad r'/2 < s < r. \]
Then the operator $S^r_{\Omega}$ is $L^p \times L^q \to L^s$ bounded with constant independent of $\Omega$, that is for every $f \in L^p$, $g \in L^q$ the inequality
\[ \Big\|\Big(\sum_{\omega \in \Omega} \Big|\iint\widehat{f}(\xi) \widehat{g}(\eta) \chi_\omega (\xi,\eta) e^{2\pi i (\xi + \eta) x} \ud \xi \ud \eta\Big|^r \Big)^{1/r}\Big\|_{L^s} \lesssim_{p,q} \|f\|_{L^p} \|g\|_{L^q} \]
holds.
\end{theorem}
The condition $p,q > r'$ is necessary, as can be seen by considering the collection of squares given by $\Omega = \{ [n,n+1] \times [m,m+1] \text{ s.t. } m, n \in \Z \}$ (indeed, in this case the bilinear operator factorizes into linear smooth Rubio de Francia operators and the necessity follows from the linear case as by \cite{TaoCowling}). The condition $1/p + 1/q = 1/s$ is equally necessary, as can be seen from a simple rescaling argument.\\
In \cite{BernicotVitturi} we extended the above result for squares to the case of non-smooth frequency projections, albeit in a smaller range. 
\begin{theorem}[\cite{BernicotVitturi}]\label{theorem_nonsmooth_squares}
Let $r > 2$ and let $\Omega \subset \mathcal{D}\times \mathcal{D}$ be a collection of disjoint squares with dyadic sides. Let $p,q,s$ be such that 
\[ \frac{1}{p} + \frac{1}{q} = \frac{1}{s} \]
and
\[ r' < p,q < r, \qquad r'/2 < s < r/2. \]
Then the operator $T^r_{\Omega}$ is $L^p \times L^q \to L^s$ bounded with constant independent of $\Omega$, that is for every $f \in L^p$, $g \in L^q$ the inequality
\[ \Big\|\Big(\sum_{\omega \in \Omega} \Big|\iint\widehat{f}(\xi) \widehat{g}(\eta) \mathds{1}_\omega (\xi,\eta) e^{2\pi i (\xi + \eta) x} \ud \xi \ud \eta\Big|^r \Big)^{1/r}\Big\|_{L^s} \lesssim_{p,q} \|f\|_{L^p} \|g\|_{L^q} \]
holds.
\end{theorem}
\begin{remark}
First of all, notice that the condition $r'/2 < s < r/2$ is redundant in the above statement: it's a consequence of the other two conditions on $p,q,s$.\\
Secondly, the larger range in Theorem \ref{theorem_smooth_squares} was obtained by means of a localisation argument (originally introduced in \cite{BeneaMuscalu}) which is unavailable in the non-smooth case.
\end{remark}
It should be remarked that the proofs of Theorems \ref{theorem_smooth_squares} and \ref{theorem_nonsmooth_squares} both rely essentially on the sets in $\Omega$ being squares.
\par 
In this paper we will extend the result above to the case where the collection consists of arbitrary dyadic \emph{rectangles} instead, again taking the frequency projections to be non-smooth. Precisely, our main result can be stated as follows.
\begin{theorem}\label{main_theorem} 
Let $r>2$ and let $\mathscr{R} \subset \mathcal{D} \times \mathcal{D}$ be a collection of disjoint dyadic \emph{rectangles} in $\widehat{\R^2}$. Let $p,q,s$ be such that 
\[ \frac{1}{p} + \frac{1}{q} = \frac{1}{s} \]
and such that 
\[ r' < p,q < r, \qquad r'/2 < s < r/2. \]
Then the operator $T^r_{\mathscr{R}}$ is $L^p \times L^q \rightarrow L^s$ bounded with constant independent of $\mathscr{R}$, that is for every $f \in L^p$, $g \in L^q$ the inequality
\[ \Big\|\Big(\sum_{R\in\mathscr{R}} |\pi_R(f,g)(x)|^r\Big)^{1/r} \Big\|_{L^s} \lesssim_{p,q} \|f\|_{L^p} \|g\|_{L^q} \]
holds.
\end{theorem}
A standard argument shows that we have as a corollary that the same result holds for smooth frequency projections instead, but this time with the added benefit of allowing \emph{arbitrary} disjoint rectangles, that is with sides not necessarily dyadic in any way (but still parallel to the axes, of course). 
\begin{corollary}\label{corollary_smooth_rectangles}
Let $r>2$ and let $\mathscr{R}$ be a collection of disjoint arbitrary rectangles in $\widehat{\R^2}$ with sides parallel to the axes. Let $p,q,s$ be such that 
\[ \frac{1}{p} + \frac{1}{q} = \frac{1}{s} \]
and such that 
\[ r' < p,q < r, \qquad r'/2 < s < r/2. \]
Then the operator $S^r_{\mathscr{R}}$ is $L^p \times L^q \rightarrow L^s$ bounded with constant independent of $\mathscr{R}$, that is for every $f \in L^p$, $g \in L^q$ the inequality
\[ \Big\|\Big(\sum_{R \in \mathscr{R}} \Big|\iint\widehat{f}(\xi) \widehat{g}(\eta) \chi_R (\xi,\eta) e^{2\pi i (\xi + \eta) x} \ud \xi \ud \eta\Big|^r \Big)^{1/r} \Big\|_{L^s} \lesssim_{p,q} \|f\|_{L^p} \|g\|_{L^q} \]
holds, where $\chi_R$ denotes
\[\chi_R(\xi, \eta) = \chi\Big(\frac{\xi - c(R_1)}{|R_1|},\frac{\eta - c(R_2)}{|R_2|}\Big).\]
\end{corollary} 
The proof of the corollary is elementary but has nevertheless been included in Appendix \ref{appendix_A} for completeness.\\
%
%  REDUCE BY DUALITY TO A TRILINEAR FORM
%
The study of the boundedness of $T_{\mathscr{R}}^r(f,g)$ is reduced by duality to the study of the boundedness of the trilinear form
\begin{equation}\label{eqn:trilinear_form_1}  
\widetilde{\Lambda}^r_{\mathscr{R}}(f,g,h) := \langle T^{r}_{\mathscr{R}}(f,g), h \rangle,  
\end{equation}
which can then be further reduced to the study of the boundedness of the trilinear form 
\begin{equation}\label{eqn:trilinear_form}  
\Lambda^r_{\mathscr{R}}(f,g,\mathbf{h}) := \int_{\R} \sum_{R\in\mathscr{R}} \pi_{R}(f,g)(x) h_R(x) \ud x,  
\end{equation}
where $\mathbf{h} = \{h_R\}_{R \in \mathscr{R}}$ is a vector valued function, specifically taking values in $\ell^{r'}(\mathscr{R})$. Observe moreover that the frequency projection $\pi_R(f,g)$ factorizes as $\pi_{R_1} f \cdot \pi_{R_2} g$, where $R = R_1 \times R_2$ and $\pi_I$ denotes the multiplier given by $\widehat{\pi_I f} = \mathds{1}_I \widehat{f}$. \par
We will prove the boundedness result stated below for the trilinear form $\Lambda^r_{\mathscr{R}}$; Theorem \ref{main_theorem} then follows from the above remarks.
%
% STATEMENT OF MAIN THEOREM FOR TRILINEAR FORM
%
\begin{theorem}\label{main_theorem_2}
Let $r>2$ and let $\mathscr{R} \subset \mathcal{D} \times \mathcal{D}$ be a collection of disjoint dyadic \emph{rectangles} in $\widehat{\R^2}$. Let $p,q,s$ be such that 
\[ \frac{1}{p} + \frac{1}{q} = \frac{1}{s} \]
and such that 
\[ r' < p,q < r, \qquad r'/2 < s < r/2. \] 
Then for every $f \in L^p$, $g \in L^q$ and $\mathbf{h} \in L^{s'}(\ell^{r'})$ we have
\begin{equation}\label{eqn:main_estimate}
\Lambda^r_{\mathscr{R}}(f,g,\mathbf{h}) \lesssim_{r,p,q} \|f\|_{L^p} \|g\|_{L^q} \|\mathbf{h}\|_{L^{s'}(\ell^{r'})}.
\end{equation}
\end{theorem}
%
%
% REDUCE TO THE CASE WHERE THE ECCENTRICITIES ARE ALL > 1
%
%
Let $\ec{R}$ denote the eccentricity of the rectangle $R$, defined as 
\[ \ec{R} := \frac{|R_2|}{|R_1|}.\]
We split the collection $\mathscr{R}$ into two subcollections, according to whether the eccentricity is high or low:
\begin{align*}
\mathscr{R}_{\mathrm{high}} &:= \{R \in \mathscr{R} \text{ s.t. } \ec{R} >1 \}, \\ 
\mathscr{R}_{\mathrm{low}} &:= \{R \in \mathscr{R} \text{ s.t. } \ec{R} \leq 1 \};
\end{align*}
of course, the rectangles of eccentricity 1 are just squares and as such we could remove them from the collection since they give rise to a bounded operator, by Theorem \ref{theorem_nonsmooth_squares}. We then split the trilinear form into $\Lambda^r_{\mathscr{R}} = \Lambda^r_{\mathscr{R}_{\mathrm{high}}} + \Lambda^r_{\mathscr{R}_{\mathrm{low}}}$. We will prove Theorem \ref{main_theorem_2} separately for $\mathscr{R}_{\mathrm{high}}, \mathscr{R}_{\mathrm{low}}$, using a time-frequency analysis of the trilinear form, and since the proof will be symmetric in the two cases we will concentrate exclusively on $\mathscr{R}_{\mathrm{high}}$ (see however Remark \ref{remark_on_trilinear_form_eccentricity_ll_1} for further details on this symmetry).
%
% strategy of proof
%
\par
Let us briefly explain the structure of the proof by analogy with the linear case of Rubio de Francia operators $\mathrm{RdF}^r_{\mathcal{I}}$, where $\mathcal{I} = \{I_n, \; n \in \Z\}$ is a given collection of disjoint intervals. Recall that the inequality $\|\mathrm{RdF}^2_{\mathcal{I}} f\|_{L^2} \leq \|f\|_{L^2}$ is a trivial consequence of Plancherel's inequality; it can be rephrased in a vector-valued flavour as $\|(\pi_{I_n} f)_{n\in \Z}\|_{L^2(\ell^2)} \leq \|f\|_{L^2}$. On the other hand, observe that pointwise $|\pi_{I_n}f(x)| \leq 2 \mathscr{C}f(x)$ independently of $n$, where $\mathscr{C}$ denotes the Carleson operator
\[ \mathscr{C}f(x) := \sup_{N \in \R} \Big| \int\limits^{N}_{-\infty} \widehat{f}(\xi) e^{2\pi i \xi x} \ud \xi\Big|, \]
and therefore we have by the Carleson-Hunt theorem that $\|(\pi_{I_n} f)_{n\in \Z}\|_{L^p(\ell^\infty)} \lesssim_p \|f\|_{L^p}$ for any $1 < p < \infty$. By complex interpolation it follows immediately that
\[ \|\mathrm{RdF}^r_{\mathcal{I}} f\|_{L^p} \lesssim_{r,p} \|f\|_{L^p} \]
for $r' < p < r$. As stated before, the operator is also bounded in the range $p \geq r > 2$, but this is a consequence of the boundedness of $\mathrm{RdF}^2_{\mathcal{I}}$ in $p\geq 2$, and the proof of this fact requires methods other than simple interpolation.\\
In the proof of Theorem \ref{main_theorem} we have adopted a similar approach. Indeed, when $r = \infty$ the bilinear operator becomes 
\[ T^\infty_{\mathscr{R}}(f,g)(x) = \sup_{R \in \mathscr{R}} |\pi_R(f,g)(x)|, \]
and this is easily seen to be bounded pointwise by $\mathscr{C}f \cdot \mathscr{C}g$, in parallel to the linear case, which in turn yields the full range of boundedness for $T^\infty_{\mathscr{R}}$. Now, the natural bilinear analogue of the $L^2 \to L^2$ estimate would be $L^2 \times L^2 \to L^1$, and thanks to an interpolation result of Silva from \cite{Silva} (see Lemma \ref{lemma_interpolation}) Theorem \ref{main_theorem_2} would follow from such an estimate for $r = 2$. However, there is no equivalent of Plancherel's theorem in the bilinear world, and this estimate is not straightforward. We currently do not know if such an estimate holds, the $r=2$ case having resisted treatment in general even in the simpler case of squares and smooth frequency projections. As a replacement, we will prove preliminary $L^p \times L^q \to L^s$ estimates for $T^r_{\mathscr{R}}$ for each $r>2$, in a range that gets arbitrarily close to $L^2 \times L^2 \to L^1$ as $r \to 2$. This will be enough to allow us to run the interpolation argument to conclude boundedness in the range stated in Theorem \ref{main_theorem}. 
\begin{remark}
The overall proof structure explained above is the same the authors have adopted in the proof of Theorem \ref{theorem_nonsmooth_squares} in \cite{BernicotVitturi}, which in turn was inspired by the previous work for smooth squares as in \cite{BeneaBernicot} and drew insights from \cite{Benea_thesis,BeneaMuscalu_preprint}, in particular the idea of using non-local operators in the stopping-time arguments. The main difference here lies in the fact that we now have to deal with the simultaneous presence of two distinct scales in each $\pi_R(f,g)$ term, namely those given by $|R_1|$ and $|R_2|$, which breaks the symmetry in how the contributions of $f$ and $g$ are treated. See particularly sections \S \ref{section_reductions}, \S \ref{section_sizes} for how this is achieved.
\end{remark}
\par
%
% OVERVIEW OF THE PROOF
%
The rest of the paper is structured as follows. In section \S\ref{section_reductions}, under the hypothesis that all rectangles have high eccentricity, we will reduce the problem to that of bounding a discretized trilinear sum with additional structure of certain (shifted) averages of $f,g,\mathbf{h}$. We will introduce a notion of time-frequency tiles adapted to the particular situation.
Sections \S\ref{section_sizes}, \S\ref{section_energies}, \S\ref{section_decomposition_lemmas} and \S\ref{section_generic_estimate} develop the time-frequency analysis tools needed to prove the main result. In particular, in section \S\ref{section_sizes} we will introduce some structured collections of tiles referred to as $\n$-columns\footnote{These play the analogous r\^{o}le played by trees, in classical time-frequency analysis terminology.}, where $\n$ denotes a shifting parameter, for which we will be able to bound the discretized trilinear sum explicitely. This process will give rise to sizes for $f$ and $\mathbf{h}$, but importantly not for $g$. Indeed, we will not try to optimize our estimates in $g$, and rather prove in section \S \ref{section:partial_results} some weaker estimates for the discretized trilinear sum in a partial range contained in $p \in (2,r)$ and $q=r$. In section \S\ref{section:proof_of_main_theorem} we will improve these estimates by interpolating the partial results with the case $r=\infty$, thus concluding the proof of Theorem \ref{main_theorem}.
%
%
% acknowledgements
%
\subsection*{Acknowledgements} Both authors are supported by ERC project FAnFArE no. 637510. The authors are also very grateful to Cristina Benea for many useful comments/discussions and in particular for having shared with us a preprint of \cite{BeneaMuscalu_preprint}.
%
%
%
%
%
%
%%%%%%%%%%%%%
%
% REDUCTION TO A DISCRETE MODEL SUM ON WAVEPACKETS
%
\section{Reduction to a discretized model sum}\label{section_reductions}
%
%
%

%
%
%  DISCRETIZATION PROCEDURE IN THE USUAL WAY
%
Recall that we are considering the collection $\mathscr{R}_{\mathrm{high}}$ of rectangles of eccentricity bigger than 1. We start by performing a standard discretization procedure on the trilinear form $\Lambda^r_{\mathscr{R}_{\mathrm{high}}}$.\\
We have (using Radon duality, with $\ud\sigma$ the induced Lebesgue measure on the plane $\xi_1 + \xi_2 + \xi_3 = 0$)
\begin{align*}
\Lambda^r_{\mathscr{R}_{\mathrm{high}}}(f,g,\mathbf{h}) & = \int_{\R} \sum_{R\in\mathscr{R}_{\mathrm{high}}} \pi_{R_1}f(x)\pi_{R_2}g(x) h_{R}(x) \ud x \\
& = \sum_{R\in\mathscr{R}_{\mathrm{high}}} \int_{\xi_1 + \xi_2 + \xi_3 = 0} \widehat{f}(\xi_1) \mathds{1}_{R_1}(\xi_1) \widehat{g}(\xi_2)\mathds{1}_{R_2}(\xi_2) \widehat{h_{R}}(\xi_3) \ud \sigma(\xi_1,\xi_2,\xi_3) \\
& = \sum_{R\in\mathscr{R}_{\mathrm{high}}} \int_{\xi_1 + \xi_2 + \xi_3 = 0} \widehat{f}(\xi_1) \mathds{1}_{R_1}(\xi_1) \widehat{g}(\xi_2)\mathds{1}_{R_2}(\xi_2) \widehat{h_{R}}(\xi_3) \chi_{R_3}(\xi_3) \ud \sigma \\
& = \sum_{R\in\mathscr{R}_{\mathrm{high}}} \int_{\R} f \ast \widecheck{\mathds{1}}_{R_1}(x) g\ast \widecheck{\mathds{1}}_{R_2}(x) h_{R} \ast \widecheck{\chi}_{R_3}(x) \ud x,
\end{align*}
where we have denoted $R_3 := 2(-R_1-R_2)$ and $\chi_{R_3}$ is a $C^\infty$ bump function adapted to $R_3$ and identically equal to $1$ on $-R_1 -R_2$. Notice that $|R_3| \sim |R_2|$ by the assumption that $R$ is of high eccentricity. Now, since the functions $f\ast \widecheck{\mathds{1}}_{R_j}$ are morally roughly constant in modulus at scale $|R_j|^{-1}$, we do the following change of variables with respect to the smallest time-scale involved, the one given by $|R_2|^{-1}$ (indeed, we are assuming $|R_2| > |R_1|$):
\begin{align*}
\sum_{R\in\mathscr{R}_{\mathrm{high}}} & \int_{\R} f \ast \widecheck{\mathds{1}}_{R_1}(x) g\ast \widecheck{\mathds{1}}_{R_2}(x) h_{R} \ast \widecheck{\chi}_{R_3}(x) \ud x \\
& = \sum_{R\in\mathscr{R}_{\mathrm{high}}} |R_2|^{-1} \int_{\R} f \ast \widecheck{\mathds{1}}_{R_1}(|R_2|^{-1}y) g\ast \widecheck{\mathds{1}}_{R_2}(|R_2|^{-1}y) h_{R} \ast \widecheck{\chi}_{R_3}(|R_2|^{-1}y) \ud y \\
& = \sum_{R\in\mathscr{R}_{\mathrm{high}}} \sum_{n \in \Z} |R_2|^{-1} \int_{0}^{1} \pi_{R_1}f(|R_2|^{-1}(n+z)) \pi_{R_2}g(|R_2|^{-1}(n+z)) \\
& \qquad \qquad\qquad \qquad \qquad \qquad \cdot h_{R} \ast \widecheck{\chi}_{R_3}(|R_2|^{-1}(n+z)) \ud z,
\end{align*}
where we have gone back to writing $\pi_\omega f$ for $f \ast \widecheck{\mathds{1}}_\omega$ in the last line.\\
% remark on two scales
\begin{remark}
Observe that $|\pi_{R_2}g|$ is morally roughly constant at scale $|R_2|^{-1}$, thus morally roughly constant on the intervals $|R_2|^{-1}(n + [0,1])$ (so is $|h_R \ast \widecheck{\chi}_{R_3}|$ since $|R_3| \sim |R_2|$ for $R \in \mathscr{R}_{\mathrm{high}}$). However, $|\pi_{R_1}f|$ is morally roughly constant at the (generally) larger scale $|R_1|^{-1}$. The presence of two different simultaneous scales is the major source of difficulty in the analysis and it is here that the difference with the proof of Theorem \ref{theorem_nonsmooth_squares} will rely. 
\end{remark}
In light of the above remark, we want to take advantage of the fact that $|\pi_{R_1}f|$ is morally roughly constant at a larger scale in space to reduce to a trilinear form with additional structure. In order to explain this reduction properly we will make use of two types of tiles (mirroring the existence of two simultaneous scales in each $\pi_R(f,g)$), the first of which we introduce now.
%
% definition of small tiles
%
\begin{definition}
A \emph{small tile} $\rho$ is a pair of the form 
\[ (\omega_2 \times I_\rho, \omega_3 \times I_\rho) \]
where $I_\rho, \omega_2$ are dyadic intervals such that there exists a rectangle $R \in \mathscr{R}$ with $\omega_2 = R_2$, $\omega_3 = R_3$ (recall $R_3 := 2(-R_1-R_2)$) and $|\omega_2||I_\rho| = 1$ (hence $|\omega_3||I_\rho| \sim 1$ as well; $\omega_3$ is in general not dyadic). Given a small tile $\rho$ we denote by $R(\rho)$ this unique rectangle $R$. The collection of all small tiles is denoted by $\SS$.
\end{definition}
Given $R \in \mathscr{R}_{\mathrm{high}}$ and $n \in \Z$ there exists a unique small tile $\rho = \rho(R,n)$ such that $R(\rho) = R$ and $I_{\rho} = |R_2|^{-1}[n,n+1]$, and viceversa. Because of this we can rewrite the trilinear form as 
\[ \Lambda^r_{\mathscr{R}_{\mathrm{high}}}(f,g,\mathbf{h}) = \sum_{\rho \in \SS} \int_{I_{\rho}} \pi_{R_1(\rho)}f(x)\pi_{R_2(\rho)}g(x) h_{R(\rho)}\ast \widecheck{\chi}_{R_3(\rho)}(x) \ud x; \]
by H\"{o}lder's inequality this is controlled by 
\[ \sum_{\rho \in \SS} \|\pi_{R_1(\rho)}f\|_{L^2(I_{\rho})} \|\pi_{R_2(\rho)}g\|_{L^2(I_{\rho})} \|h_{R(\rho)}\ast \widecheck{\chi}_{R_3(\rho)}\|_{L^\infty(I_{\rho})}. \]
Now we will use the fact that $|\pi_{R_1}f|$ is roughly constant on many $I_{\rho}$'s. Let $\rho$ be fixed and let $n$ be such that $I_{\rho} = |R_2|^{-1}[n,n+1]$. There exists a unique $k \in \Z$ such that $n = k \cdot \ec{R} + \ell$ with $0 \leq \ell < \ec{R}$, and thus we can associate uniquely to $\rho$ the interval $I = |R_1|^{-1}[k,k+1]$. This is a dyadic interval of length $|R_1|^{-1}$ and it has the property that $I_{\rho} \subset I$.\\
Given the interval $\omega$ let $\chi_\omega$ denote a smooth function such that $\chi_\omega (\xi) = 1$ for $\xi \in \omega$ and $\chi_\omega(\xi) = 0$ for $\xi \not\in 2\omega$. By definition of the frequency projections we then have $\widecheck{\chi}_{R_1} \ast (\pi_{R_1}f) = \pi_{R_1}f$, and thus we can bound for every $N > 0$
\[ |\pi_{R_1(\rho)}f(x)| \lesssim_N \int |\pi_{R_1(\rho)}f(y)| \frac{1}{|R_1|^{-1}}\Big(1 + \frac{|x-y|}{|R_1|^{-1}}\Big)^{-N} \ud y.  \]
Therefore we see that we can bound (by Minkowski's inequality)
\begin{align*}
\frac{\|\pi_{R_1(\rho)}f\|_{L^2(I_{\rho})}}{|I_{\rho}|^{1/2}} \lesssim & \int |\pi_{R_1(\rho)}f(y)| \sum_{\n \in \Z} (1 + |\n|)^{-N} \frac{\mathds{1}_{I^{\n}}(y)}{|I|} \ud y \\
& = \sum_{\n \in \Z} (1 + |\n|)^{-N} \fint_{I^\n} |\pi_{R_1(\rho)}f|,
\end{align*}
where $I$ is as above (in particular $I \supset I_\rho$) and 
\[ I^{\n} := I + \n|I|, \]
that is $I^{\n}$ denotes the interval $I$ shifted by $\n$ times its length.\\
With this in mind, we now define the second type of tiles.
%
% definition of super tiles
%
\begin{definition}
Given a rectangle $R \in \mathscr{R}$ and a dyadic interval $I$, we define the (possibly empty) collection of small tiles $\SS^{\n}_{R,I}$ to be
\[ \SS^{\n}_{R,I} := \{ \rho = (\omega_2 \times I_\rho, \omega_3 \times I_\rho) \text{ small tile : } \omega_2 = R_2,\; \omega_3 = R_3,\; I_\rho \subseteq I^{\n} \}. \]
A \emph{super tile} $P$ is a pair of the form 
\[ (R_1 \times I,\; \SS^{\n}_{R, I}) \]
where $R \in \mathscr{R}$ and $I$ is a dyadic interval such that $|R_1||I| = 1$ (notice that the pair $(R,I)$ completely determines the super tile).\\
Given $P$ as above, we let $P_1 := R_1 \times I$ and call it simply a \emph{tile}, $I_P := I$, $R(P) := R$ (and analogously $R_j(P) := R_j$ for $j \in \{1,2,3\}$) and $\SS^{\n}_P := \SS^{\n}_{R, I}$. Finally, if $\P$ is a collection of super tiles, we let $\mathscr{R}(\P) := \{ R(P) \text{ : } P \in \P \}$.
\end{definition}
Now fix a parameter $r_0$ such that $r > r_0 > 2$. This will remain fixed throughout sections \S\ref{section_sizes} - \S\ref{section:partial_results} (ultimately we will take $r_0$ to depend on $r$; one could take for example $r_0 = (r+2)/2$). Using the fact that $\fint_{I} |\pi_{R_1}f| \leq \big(\fint_{I} |\pi_{R_1}f|^{r_0}\big)^{1/r_0}$ and the above remarks we can finally bound 
\[ |\Lambda^r_{\mathscr{R}_{\mathrm{high}}}(f,g,\mathbf{h})| \lesssim_{N} \sum_{\n \in \Z} (1 + |\n|)^N \Lambda^{\n}_{\P}(f,g,\mathbf{h}), \]
where $\Lambda^{\n}_{\P}$ denote the \emph{shifted} trilinear forms
\begin{equation}\label{eqn:trilinear_form_with_tiles} 
\begin{aligned}
\Lambda^{\n}_{\P}(f,g,\mathbf{h}) := \sum_{P \in \P} \big(\fint_{I_P} |\pi_{R_1(P)}f|^{r_0}\big)^{1/r_0}&  \Big[
\sum_{\rho \in \SS^{\n}_P} |I_\rho|^{1/2} \|\pi_{R_2(P)}g\|_{L^2(I_{\rho})} \\
& \times  \|h_{R(P)}\ast \widecheck{\chi}_{R_3(P)}\|_{L^\infty(I_{\rho})}\Big].  
\end{aligned}
\end{equation}
To bound $\Lambda^r_{\mathscr{R}_{\mathrm{high}}}(f,g,\mathbf{h})$ it will therefore suffice to bound the trilinear forms $\Lambda^{\n}_{\P}(f,g,\mathbf{h})$ in the same range with a constant that is at most polynomial in $\n$ (indeed, we will be able to show it is at most logarithmic).\\
In the following sections we will consider $\n$ fixed and study the trilinear forms $\Lambda^{\n}_{\P}$.
%
% remark on introducing r_0
%
\begin{remark}
It might seem that we are introducing a gratuitous inefficiency in our argument by replacing the $L^1$-average $\fint_{I} |\pi_{R_1}f|$ with the larger $\big(\fint_{I} |\pi_{R_1}f|^{r_0}\big)^{1/r_0}$, but this will not be the case. This is due to our use of interpolation in the proof of Theorem \ref{main_theorem_2}: for the argument to carry through we will only need to prove boundedness of $\Lambda^r_{\mathscr{R}_{\mathrm{high}}}$ in any range sufficiently close to the  $L^2 \times L^2 \to L^1$ estimate, for all given $r>2$ (see Section \ref{section:proof_of_main_theorem}). Using $L^1$-averages would not enlarge the range obtained through interpolation, at least not with our argument.
\end{remark}
%
%%%%%%%%%%%%%%%%%%%%%%%%%%%%%%%%%%%%%
%
% SECTION ON n-COLUMNS AND SIZES
%
\section{$\n$-Columns and sizes}\label{section_sizes}
We will now introduce some special collections of super tiles for which we will be able to bound the (shifted) trilinear form explicitely. In the process, we will come up with quantities that will be used as sizes in order to perform a time-frequency analysis of the general trilinear form.\par
First we define an order relation on the super tiles.
%
% DEFINITION OF ORDER OF SUPER TILES
%
\begin{definition}\label{definition_order_super_tiles}
Let $P, P'$ be super tiles and let $\n$ be fixed. Then we say that $P \prec_{\n}P'$ if 
\begin{align*}
R_1(P) & \supset R_1(P'), \\
I^{\n}_P & \subset I^{\n}_{P'}.
\end{align*}
We say that $P \preceq_{\n} P'$ if $P \prec_{\n}P'$ or $P = P'$.
\end{definition}
%
% DEFINITION OF THE n-COLUMNS
%
\begin{definition}
An \emph{$\n$-column} $\C$ with top a super tile $P_{\mathrm{top}}$ is a collection of super tiles such that for every $P \in \C$
\[ P \preceq_{\n} P_{\mathrm{top}}.\]
We denote $\mathsf{Top}(\C) = P_{\mathrm{top}}$.
\end{definition}
\begin{remark}
Notice the r\^{o}le of $\n$ in the above definitions. The difference between the above $\n$-columns and the columns originally defined in \cite{BeneaBernicot} resides in the fact that there is now the shifting parameter $\n$ to be taken into account. Thus, unlike in \cite{BeneaBernicot}, the tiles $P_1$ for $P \in \C$, that is the sets $R_1(P)\times I_P$, do not form an overlapping tree - unless $\n = 0$. See figure \ref{figure:n_column} for further clarification.
\end{remark}

The following lemma regarding the structure of an $\n$-column is essentially obvious once one unpacks the definitions; it will find its application in the discussion that follows.
\begin{lemma}\label{lemma_column_structure}
Let $\C$ be a column with top $P_{\mathrm{top}}$. Then for every super tile $P \in \C$ we have 
\[ \forall \rho \in \SS^{\n}_P, \quad I_\rho \subset I^{\n}_{P_{\mathrm{top}}}. \]
Moreover, if $R, R' \in \mathscr{R}(\C)$ are distinct, then we have $R_2 \cap R'_2 = \emptyset$.
\end{lemma}
In other words, the small tiles belonging to an $\n$-column agree with the time-frequency portrait of a column as in \cite{BeneaBernicot}. The second property in the lemma is due to the disjointness of the rectangles in $\mathscr{R}$. See figure \ref{figure:n_column} for a pictorial representation of an $\n$-column. 
%
%%%%%%%%%%%%%%%%%%%%%%%%%%%%%
%
% FIGURE REPRESENTING AN n-COLUMN
%
\begin{figure}[ht]
\centering
\input{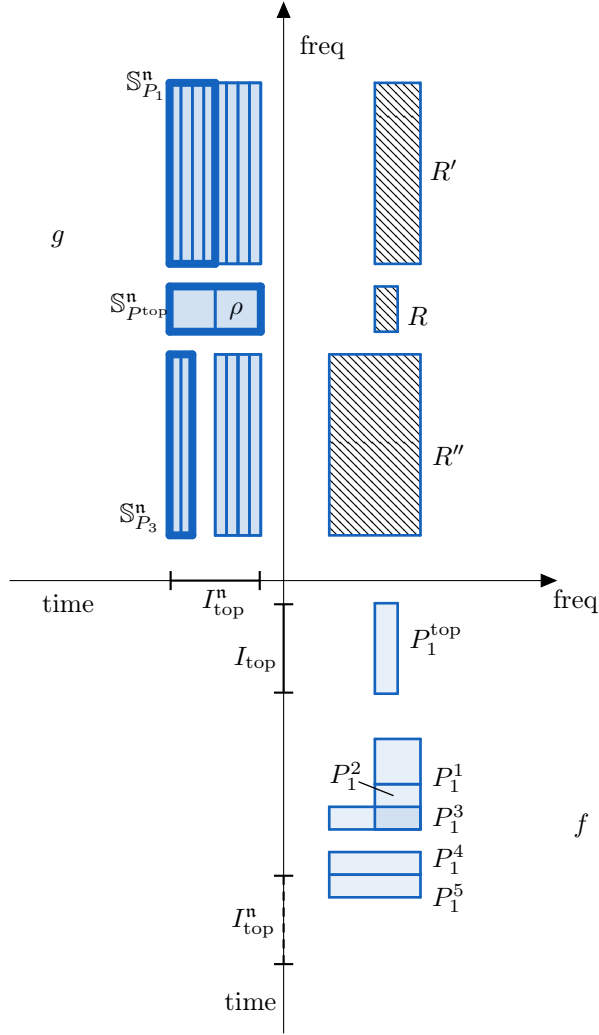}
\caption{\footnotesize A composite pictorial representation of an $\n$-column $\C$ for $\n = 3$. Quadrant I represents the frequency plane, quadrant II represents the time-frequency picture of $g$, quadrant IV represents the (rotated) time-frequency picture of $f$. The column consists of supertiles $\{P^{\mathrm{top}}, P^1, \ldots, P^5\}$ and has top $P^{\mathrm{top}}$, whose time interval is $I_{\mathrm{top}}$. The dashed rectangles in quadrant I are the frequency rectangles: $\mathscr{R}(\C) = \{R,R',R''\}$. The rectangles in quadrant II are small tiles (each has area 1). The thicker rectangles in quadrant II represent the union of all small tiles in $\SS^{\n}_P$ for $P = P^{\mathrm{top}}, P^1, P^3$. The rectangles in quadrant IV are the tiles $P^{\mathrm{top}}_1, P^1_1, \ldots, P^5_1$ as labeled (each has area 1). Observe that $R(P^{\mathrm{top}}) = R$, $R(P^1)=R(P^2)=R'$, $R(P^3)=R(P^4)=R(P^5)=R''$. Notice also that for any $j=1, \ldots, 5$ we have $I^{\n}_{P^j_1} \subset I^{\n}_{\mathrm{top}}$, as per definition.} \label{figure:n_column}
\end{figure}\\
%
%%%%%%%%%%%%%%%%%%%%%%%%%%%
%
Let then $\C$ be an $\n$-column and consider the trilinear form $\Lambda^{\n}_{\C}(f,g,\mathbf{h})$. First of all, we fix a super tile $P$ in $\C$ and look at the inner sum in the small tiles $\rho$ in \eqref{eqn:trilinear_form_with_tiles}: we have by H\"{o}lder inequality 
\begin{align*}
\sum_{\rho \in \SS^{\n}_P} & |I_\rho|^{1/2} \|\pi_{R_2(P)}g\|_{L^2(I_{\rho})} \|h_{R(P)}\ast \widecheck{\chi_{R_3(P)}}\|_{L^\infty(I_{\rho})} \\
& \leq 
\Big(\sum_{\rho \in \SS^{\n}_P} |I_\rho|^{1-r/2}  \|\pi_{R_2(P)}g\|_{L^2(I_{\rho})}^{r} \Big)^{1/r} 
\Big(\sum_{\rho \in \SS^{\n}_P} |I_\rho| \|h_{R(P)}\ast \widecheck{\chi_{R_3(P)}}\|_{L^\infty(I_{\rho})}^{r'} \Big)^{1/{r'}}.
\end{align*}
For the term in $g$, by Lemma \ref{lemma_column_structure} we can bound by H\"{o}lder's inequality
\[ \sum_{\rho \in \SS^{\n}_P} |I_\rho|^{1-r/2}  \|\pi_{R_2(P)}g\|_{L^2(I_{\rho})}^{r} = \sum_{\rho \in \SS^{\n}_P} |I_\rho| \Big(\fint_{I_{\rho}} |\pi_{R_2(P)}g|^2\Big)^{r/2} \leq \int_{I^{\n}_P} |\pi_{R_2(P)}g|^r, \]
because $r > 2$. \\
For the term in $\mathbf{h}$ instead, we see that we have for any small tile $\rho$ and any large $M > 0$
\begin{align*}
|I_\rho| \|h_{R(P)}\ast \widecheck{\chi_{R_3(P)}}\|_{L^\infty(I_{\rho})}^{r'} & = |I_\rho| \sup_{y \in I_\rho} |h_{R}\ast \widecheck{\chi_{R_3}}(y)|^{r'} \\
& \leq  |I_\rho|\big(\sup_{y \in I_\rho} \int |h_{R}(z)| |\widecheck{\chi_{R_3}}(y-z)| \ud z  \big)^{r'} \\
& \lesssim |I_\rho| \Big(\sup_{y \in I_\rho} \int |h_{R}(z)| \big(1 + \frac{|y-z|}{|I_\rho|}\big)^{-M} \frac{\ud z}{|I_\rho|}  \Big)^{r'} \\
& \leq |I_\rho| \big(\int |h_{R}(z)| \sup_{y \in I_\rho} \big(1 + \frac{|y-z|}{|I_\rho|}\big)^{-M} \frac{\ud z}{|I_\rho|}  \big)^{r'} \\
& \lesssim \int |h_{R}(z)|^{r'} \Phi_{I_\rho}(z) \ud z, 
\end{align*}
where $\Phi_{I}$ denotes a rapidly decaying function concentrated in $I$ and equal to $1$ there. In particular, $\Phi_I$ decays like $(1 + \mathrm{dist}(z,I)/|I|)^{-N}$ for some large $N$. Therefore, summing over $\rho \in \SS^{\n}_P$, we have 
\[ \sum_{\rho \in \SS^{\n}_P} |I_\rho| \|h_{R(P)}\ast \widecheck{\chi_{R_3(P)}}\|_{L^\infty(I_{\rho})}^{r'} \lesssim \int |h_{R(P)}(z)|^{r'} \Phi_{I^{\n}_P}(z) \ud z. \]
With these preliminary estimates at hand, we proceed to estimate $\Lambda^{\n}_{\C}(f,g,\mathbf{h})$. Write $I_{\mathrm{top}}$ in place of $I_{\mathsf{Top}(\C)}$ for shortness. We take the supremum in the terms in $f$ and use H\"{o}lder's inequality with exponents $(r,r')$ once again, this time in the super tiles, thus obtaining
\begin{align*}
\Lambda^{\n}_{\C}(f,g,\mathbf{h}) \lesssim &  \sup_{P \in \mathcal{C}}\Big(\fint_{I_P} |\pi_{R_1(P)}f|^{r_0} \Big)^{1/r_0} \Big[ \frac{1}{|I_{\mathrm{top}}|} \sum_{P \in \mathcal{C}}\int_{I^{\n}_P} |\pi_{R_2(P)}g|^r \Big]^{1/r}  \\
& \qquad \times \Big(\frac{1}{|I_{\mathrm{top}}|} \sum_{P \in \mathcal{C}}\int |h_{R(P)}|^{r'} \Phi_{I^{\n}_P} \Big)^{1/{r'}} |I_{\mathrm{top}}|.
\end{align*}
We then define sizes for $f$ and $\mathbf{h}$ according to the above estimate. 

\begin{definition}[Sizes]
%
% DEFINITION OF SIZE FOR f
%
Let $\P$ be a collection of super tiles, and let $f \in L^1_{\mathrm{loc}}(\R)$, $\mathbf{h} \in L^1_{\mathrm{loc}}(\ell^{r'})$. Then we define 
\[ \sizen_f(\P) := \sup_{P \in \P} \Big(\fint_{I_P} |\pi_{R_1(P)}f|^{r_0} \Big)^{1/r_0}, \]
and 
%
% DEFINITION OF SIZE FOR h
%
\[ \sizen_{\mathbf{h}}(\P) := \sup_{\substack{\C \text{ $\n$-column}, \\ \C \subset \P}} \Big(\frac{1}{|I_{\top{\C}}|} \sum_{P \in \mathcal{C}}\int |h_{R(P)}|^{r'} \Phi_{I^{\n}_P} \Big)^{1/{r'}}. \]
\end{definition}
Notice that, strictly speaking, $\sizen_f$ does not depend on the shifting parameter $\n$. However, the associated energy will and thus we keep the index $\n$ as a reminder.\\
We can then summarise the above discussion in the following proposition.
%
% PROPOSITION ON ESTIMATE FOR A SINGLE COLUMN
%
\begin{proposition}\label{proposition_column_control}
Let $\C$ be an $\n$-column. Then we can bound 
\[ \Lambda^{\n}_{\C}(f,g,\mathbf{h}) \lesssim  \Big[\frac{1}{|I_{\top{\C}}|} \sum_{P \in \C} \int_{I^{\n}_P} |\pi_{R_2(P)}g|^r  \Big]^{1/r} \\
 \sizen_f(\C) \sizen_{\mathbf{h}}(\C) |I_{\top{\C}}|. \]
\end{proposition}
As explained in the overview of the proof, we will not introduce a size for $g$. Control of the contribution of $g$ will rather be achieved through the following lemma.
%
% LEMMA ON CONTROLLING g BY AN L^2 AVERAGE
%
\begin{lemma}\label{lemma_g_control}
If $\C$ is an $\n$-column, then we have 
\[ \frac{1}{|I_{\top{\C}}|} \sum_{P \in \C} \int_{I^{\n}_P} |\pi_{R_2(P)}g|^r  \lesssim  \fint_{I^{\n}_{\top{\C}}} (\VarC{r}g(x))^r \ud x, \]
where $\VarC{r}$ is the Variational Carleson operator given by
\[ \VarC{r}g(x) := \sup_N \sup_{\xi_1 < \cdots < \xi_N} \Big( \sum_{j=1}^{N-1} |\pi_{[\xi_j, \xi_{j+1}]} g(x)|\Big)^{1/r}. \]
\end{lemma}
\begin{proof}
By the definition of $\n$-column, we have that $I^{\n}_P \subset I^{\n}_{\top{\C}}$ for all $P \in \C$, and moreover if $P, P' \in \C$ are such that $R_2(P) \neq R_2(P')$ we necessarily have $R_2(P) \cap R_2(P') = \emptyset$ (see Lemma \ref{lemma_column_structure}). In other words, the ``shifted tiles'' $R_2(P) \times I^{\n}_P$ are all disjoint and contained in the strip $\R \times I^{\n}_{\top{\C}}$. If we then rewrite 
\[ \frac{1}{|I_{\top{\C}}|} \sum_{P \in \C} \int_{I^{\n}_P} |\pi_{R_2(P)}g|^r = \frac{1}{|I_{\top{\C}}|} \int_{I^{\n}_{\top{\C}}}\sum_{P \in \C}  |\pi_{R_2(P)}g(x)|^r \mathds{1}_{I^{\n}_P}(x) \ud x  \]
we see that we can bound the integrand pointwise by $(\VarC{r}g(x))^r$, and we are done.
\end{proof}
\begin{remark}
The above lemma relies crucially on the structure of the $\n$-columns to hold; thus we can see that the lemma provides motivation for our definition.
\end{remark}
Finally, we conclude this section with the observation that the sizes introduced are controlled by maximal averages, which will allow us to perform the time-frequency analysis. 
\begin{lemma}\label{lemma_f_control}
Let $\P$ be a collection of super tiles. Then we have 
\[ \sizen_f(\P) \lesssim \sup_{P \in \P} \Big(\fint_{I_P} |\mathscr{C}f|^{r_0}\Big)^{1/r_0}, \]
where $\mathscr{C}$ is the Carleson operator.
\end{lemma}
\begin{proof}
Obvious.
\end{proof}
\begin{lemma}\label{lemma_h_control}
Let $\P$ be a collection of super tiles. Then we have 
\[ \sizen_{\mathbf{h}}(\P) \lesssim \sup_{P \in \P} \Big(\frac{1}{|I_P|} \int \|\mathbf{h}\|_{\ell^{r'}}^{r'} \Phi_{I^{\n}_P}  \Big)^{1/{r'}}. \]
\end{lemma}
\begin{proof}
The claim follows quickly from the definition of an $\n$-column and the rapid decay of the functions $\Phi_{I^{\n}_P}$.
\end{proof}
%
%
%%%%%%%%%%%%
%
% INTRODUCING ENERGIES AND PROVING THE ENERGY ESTIMATES
%
\section{Energies}\label{section_energies}
In this section we will introduce the remaining quantities that we need in order to enable the time-frequency analysis of $\Lambda^{\n}_{\P}$. It will be necessary to introduce a special notion of disjointness for $\n$-columns (quite similar to that in \cite{BeneaBernicot}) in order to ensure good control of the quantities we are going to introduce, which is as follows.
%
% DEFINITION OF MUTUAL DISJOINTNESS OF \n-COLUMNS
%
\begin{definition}\label{definition_mutual_disjointness}
Let $\mathfrak{C}$ be a collection of $\n$-columns. We say that the $\n$-columns in $\mathfrak{C}$ are \emph{mutually disjoint} if
\begin{enumerate}[i)]
\item for every $\C, \C' \in \mathfrak{C}$ we have $\C \cap \C' = \emptyset$ (that is, the $\n$-columns are disjoint as sets of super tiles);
\item any two tops are not comparable under the $\preceq_{\n}$ relation, or equivalently the sets $R_1(\top{\C}) \times I^{\n}_{\top{\C}}$ for $\C \in \mathfrak{C}$ are pairwise disjoint.
\end{enumerate}
\end{definition}
Notice the presence of the shifting parameter $\n$ in the above definition. It is not the tiles $\top{\C}_1$ that are assumed to be disjoint sets in time-frequency, but rather their shifted versions.\\
With this notion of mutual disjointness we can now define the energies. 
%
% ENERGY OF f
%
\begin{definition}[Energy of $f$]
Let $\P$ be a collection of super tiles, and let $f \in L^1_{\mathrm{loc}}(\R)$. We define 
\[ \energyn_f (\P) := \sup_{n \in \Z} \sup_{\mathfrak{C}} 2^{n} \Big(\sum_{\C \in \mathfrak{C}} |I_{\top{\C}}|\Big)^{1/r_0}, \] 
where the inner supremum ranges over all collections $\mathfrak{C}$ of mutually disjoint \mbox{$\n$-columns} $\C$ such that for each $\C \in \mathfrak{C}$ and each $P \in \C$
\[ \Big(\fint_{I_P} |\pi_{R_1(P)}f|^{r_0}\Big)^{1/r_0} \geq 2^n. \]
\end{definition}
%
% ENERGY OF h
%
\begin{definition}[Energy of $\mathbf{h}$]
Let $\P$ be a collection of super tiles, and let $\mathbf{h} \in L^1_{\mathrm{loc}}(\ell^{r'})$. We define 
\[ \energyn_{\mathbf{h}} (\P) := \sup_{n \in \Z} \sup_{\mathfrak{C}} 2^{n} \Big(\sum_{\C \in \mathfrak{C}} |I_{\top{\C}}|\Big)^{1/{r'}}, \] 
where the inner supremum ranges over all collections $\mathfrak{C}$ of mutually disjoint \mbox{$\n$-columns} $\C$ such that for each $\C \in \mathfrak{C}$
\[ \Big(\frac{1}{|I_{\top{\C}}|} \sum_{P \in \mathcal{C}}\int |h_{R(P)}|^{r'} \Phi_{I^{\n}_P} \Big)^{1/{r'}} \geq 2^n. \]
\end{definition}
\begin{remark}
Observe that $\energyn_f$ is an $L^{r_0}$ quantity, while $\energyn_{\mathbf{h}}$ is an $L^{r'}$ quantity. The fact that $r' < 2 < r_0$ will be fundamental in summing up the energies coming from a decomposition of $\P$ into $\n$-columns (see Proposition \ref{proposition_generic_estimate}).
\end{remark}
We now show how to control the energies in terms of $L^p$ norms of the functions $f, \mathbf{h}$. We start with the latter.
%
% ENERGY LEMMA FOR h
%
\begin{lemma}\label{lemma_energy_h}
Let $\P$ be a collection of super tiles and let $\mathbf{h} \in L^{r'}(\ell^{r'})$. Then we have 
\[ \energyn_{\mathbf{h}}(\P) \lesssim \|\mathbf{h}\|_{L^{r'}(\ell^{r'})}. \]
\end{lemma}
\begin{proof}
This lemma is essentially the same as Proposition 2.10 in \cite{BernicotVitturi}.\\
Let $n \in \Z$ and $\mathfrak{C}$ be a collection of mutually disjoint $\n$-columns such that the pair realizes the supremum in the definition of $\energyn_{\mathbf{h}}$ within a factor of 2, that is 
\[ \energyn_{\mathbf{h}}(\P)^{r'} \sim 2^{r'n} \sum_{\C \in \mathfrak{C}} |I_{\top{\C}}|. \]
Then the right hand side is dominated by 
\[ \sum_{\C \in \mathfrak{C}} \sum_{P \in \mathcal{C}}\int |h_{R(P)}|^{r'} \Phi_{I^{\n}_P}, \]
which in turn is dominated by 
\[ \sum_{R \in \mathscr{R}} \int |h_{R}|^{r'}\Big( \sum_{\C \in \mathfrak{C}}  \sum_{\substack{ P \in \C : \\ R(P) = R}}\Phi_{I^{\n}_{P}} \Big). \]
Observe that as $P$ ranges over the super tiles such that $R(P) = R$, the intervals $I^{\n}_P$ are all disjoint; as the $\n$-columns are mutually disjoint we then have that  
\[ \sum_{\C \in \mathfrak{C}}  \sum_{\substack{ P \in \C : \\ R(P) = R}}\Phi_{I^{\n}_{P}} \lesssim 1 \]
and the conclusion follows.
\end{proof}
Next we will show how to control $\energyn_f$ in terms of $\|f\|_{L^{r_0}}$. Here the presence of the shift parameter will produce an unavoidable logarithmic loss in $\n$ in the inequality, which will however be more than acceptable for our purposes.
%
% ENERGY LEMMA FOR f
%
\begin{lemma}\label{lemma_energy_f}
Let $\P$ be a collection of super tiles and let $f \in L^{r_0}(\R)$. Then 
\[ \energyn_f(\P) \lesssim \log^{+}(\n) \|f\|_{L^{r_0}}, \]
where $\log^{+}(x) = \log(2 + |x|)$.
\end{lemma}
\begin{proof}
Let $n \in \Z$ and $\mathfrak{C}$ be a collection (which we can assume to be finite) of $\n$-columns such that they realize the supremum in the definition of $\energyn_f$ within a factor of 2, that is 
\[ \energyn_f(\P)^{r_0} \sim 2^{r_0 n} \sum_{\C \in \mathfrak{C}} |I_{\top{\C}}|. \]
We'd like to argue by pointwise control by a Variational Carleson operator as in Lemma \ref{lemma_g_control}, but we can't do so at this stage because, by the nature of the definition of mutually disjoint $\n$-columns, the tops $\top{\C}_1$ are not disjoint in general. However, we will show that up to introducing a loss of $\log^{+}(\n)$ in the estimate, we can reduce to the disjoint case. To do so it is convenient to introduce a second ordering on the super tiles (distinct from the previously defined order $\prec_{\n}$).
\begin{definition}
Let $P, P'$ be super tiles. We say that $P < P'$ if 
\begin{align*}
R_1(P) &\supsetneq R_1(P'), \\
I_P & \subsetneq I_{P'}.
\end{align*}
\end{definition}
Thus, the difference between the $\prec_{\n}$ order relation and the newly introduced $<$ is that the former involves the shifting parameter $\n$, as the subscript suggests, while the latter doesn't. 
Now, let 
\[ \P_{\mathrm{tops}} = \{ \top{\C} \text{ s.t. }\C \in \mathfrak{C}\} \]
and let $\P_{\mathrm{tops}}^{\mathrm{max}}$ denote the subcollection of super tiles that are maximal with respect to the ordering $<$ just introduced. Notice that the tiles in $\P_{\mathrm{tops}}$ are all distinct by assumption, which has the following consequence: if one fixes $P_0 \in \P_{\mathrm{tops}}^{\mathrm{max}}$ then trivially
\[ \sum_{\substack{ P \in \P_{\mathrm{tops}} \;: P < P_0, \\ |I_P|\gtrsim \n^{-1}|I_{P_0}|} } |I_P| \lesssim \log^{+}(\n) |I_{P_0}|. \]
As for the tiles $P \in \P_{\mathrm{tops}}$ such that $P < P_0$ but $|I_P| \ll \n^{-1} |I_{P_0}|$, observe that the latter implies that $I^{\n}_P \subset 3 I_{P_0}$. Since the $\n$-columns were assumed to be mutually disjoint, the shifted tiles $R_1(P) \times I^{\n}_P$ are all disjoint, and therefore 
\[ \sum_{\substack{ P \in \P_{\mathrm{tops}} \;: P < P_0, \\ |I_P|\ll \n^{-1}|I_{P_0}|} } |I^{\n}_P| \lesssim |I_{P_0}|. \]
We have thus shown that
\[ \sum_{P \in \P_{\mathrm{tops}}} |I_P| \lesssim \log^{+}(\n) \sum_{P_0 \in \P_{\mathrm{tops}}^{\mathrm{max}}} |I_{P_0}|. \]
Since we have by assumption
\[\Big(\fint_{I_P} |\pi_{R_1(P)}f|^{r_0}\Big)^{1/r_0} \geq 2^n \]
for all $P \in \P_{\mathrm{tops}}$, we are therefore able to bound 
\[ \energyn_f(\P)^{r_0} \lesssim  \log^{+}(\n)  \int \sum_{P_0 \in \P_{\mathrm{tops}}^{\mathrm{max}}} |\pi_{R_1(P_0)}f|^{r_0} \mathds{1}_{I_{P_0}}; \]
but by maximality the super tiles in $\P_{\mathrm{tops}}^{\mathrm{max}}$ are all disjoint (as tiles), and therefore we can bound pointwise (as in Lemma \ref{lemma_g_control})
\[ \sum_{P_0 \in \P_{\mathrm{tops}}^{\mathrm{max}}} |\pi_{R_1(P_0)}f(x)|^{r_0} \mathds{1}_{I_{P_0}}(x) \leq (\VarC{r_0}f(x))^{r_0}. \]
Thus we have proven 
\[ \energyn_f(\P)^{r_0} \lesssim  \log^{+}(\n) \int (\VarC{r_0}f(x))^{r_0} \ud x, \]
and the lemma follows (even with the smaller constant $O(\log^{+}(\n)^{1/r_0})$) from the fact that $r_0 > 2$ and hence $\VarC{r_0}$ is $L^{r_0} \to L^{r_0}$ bounded (see \cite{OberlinSeegerTaoThieleWright}).
\end{proof}
\begin{remark}
The argument in the above lemma motivates our choice of introducing the parameter $r_0 > 2$. Indeed, the Variational Carleson operator $\VarC{q}$ is unbounded when $q \leq 2$, as shown in \cite{OberlinSeegerTaoThieleWright}.
\end{remark}
%
%
%%%%%%%%%%%%
%
% DECOMPOSITION LEMMAS 
%
\section{Decomposition lemmas}\label{section_decomposition_lemmas}
Now that all the relevant quantities are in place, we will establish two decomposition lemmas that will allow us to partition every collection of super tiles $\P$ into structured subcollections which have controlled size and energy. The results in this section are classical and are based on simple stopping time arguments.
%
% DECOMPOSITION LEMMA FOR f
% 
\begin{lemma}[Decomposition lemma for $f$]\label{lemma_decomposition_f}
Let $\P$ be a collection of super tiles, $f \in L^1_{\mathrm{loc}}(\R)$, and let $n \in \Z$ be such that 
\[ \sizen_f(\P) \leq 2^{-n} \energyn_f(\P). \]
Then we can decompose $\P$ into $\P_{\mathrm{high}} \sqcup \P_{\mathrm{low}}$ so that 
\[ \sizen_f(\P_{\mathrm{low}}) \leq \frac{1}{2} 2^{-n} \energyn_f(\P), \]
and $\P_{\mathrm{high}}$ can be partitioned into a collection $\mathfrak{C}$ of mutually disjoint $\n$-columns such that 
\[ \sum_{\C \in \mathfrak{C}} |I_{\top{\C}}| \lesssim 2^{r_0 n}. \]
\end{lemma}
\begin{proof}
We can assume the collection $\P$ is finite, for simplicity. Let $\P_{\mathrm{stock}}$ be initialized to 
\[ \P_{\mathrm{stock}} := \Big\{ P \in \P \text{ s.t. } \Big(\fint_{I_P} |\pi_{R_1(P)}f|^{r_0}\Big)^{1/r_0}> 2^{-n-1}\energyn_{f}(\P) \Big\}, \]
and let $\mathfrak{C}$ be initialized to $\mathfrak{C} := \emptyset$. We set right away 
\[ \P_{\mathrm{low}} := \P \backslash \P_{\mathrm{stock}}, \]
which will not be changed throughout the algorithm. The size property is then immediate from the definition of (the initial state of) $\P_{\mathrm{stock}}$. As for the organization of $\P_{\mathrm{high}} := \P \backslash \P_{\mathrm{low}}$ into $\n$-columns, we proceed as follows. Let $P_{\mathrm{max}}$ be the maximal super tile in $\P_{\mathrm{stock}}$ (with respect to $\prec_{\n}$) such that $\inf I^{\n}_{P_{\mathrm{max}}}$ is  minimum and $\sup R_1(P_{\mathrm{max}})$ is maximum (any other total order will do, of course). Then we let $\mathcal{C}$ be the maximal $\n$-column (with respect to $\prec_{\n}$) in $\P_{\mathrm{stock}}$ with top $P_{\mathrm{max}}$, update $\mathfrak{C}$ to be $\mathfrak{C} \cup \{\mathcal{C}\}$ and update $\P_{\mathrm{stock}}$ to be $\P_{\mathrm{stock}}\backslash \bigcup_{P \in \mathcal{C}}\{P\}$. Repeat the process until $\P_{\mathrm{stock}}$ is empty and the algorithm stops. Then we see that by maximality the $\n$-columns in $\mathfrak{C}$ are mutually disjoint, and as for the bound on $\sum_{\mathcal{C} \in \mathfrak{C}}|I_{\mathcal{C}}|$ notice that we have for each $\mathcal{C}$ and for each $P \in \mathcal{C}$
\[ \Big(\fint_{I_P} |\pi_{R_1(P)}f|^{r_0}\Big)^{1/r_0} > 2^{-n-1} \energyn_{f}(\P)\]
and therefore just by definition of energy (and its monotonicity)
\[ 2^{-n-1} \energyn_{f}(\P) \Big(\sum_{\mathcal{C}\in\mathfrak{C}} |I_{\mathcal{C}}|\Big)^{1/r_0} \lesssim \energyn_{f}(\P_{\mathrm{high}}) \leq \energyn_{f}(\P), \]
which proves the claim.  
\end{proof}
Next we have
%
% DECOMPOSITION LEMMA FOR h
%
\begin{lemma}[Decomposition lemma for $\mathbf{h}$]\label{lemma_decomposition_h}
Let $\P$ be a collection of super tiles, $\gamma = 2^{r_0/{r'}}$, $\mathbf{h} \in L^1_{\mathrm{loc}}(\ell^{r'})$, and let $n \in \Z$ be such that 
\[ \sizen_{\mathbf{h}}(\P) \leq \gamma^{-n} \energyn_{\mathbf{h}}(\P). \]
Then we can decompose $\P$ into $\P_{\mathrm{high}} \sqcup \P_{\mathrm{low}}$ so that 
\[ \sizen_{\mathbf{h}}(\P_{\mathrm{low}}) \leq \gamma^{-n-1} \energyn_{\mathbf{h}}(\P), \]
and $\P_{\mathrm{high}}$ can be partitioned into a collection $\mathfrak{C}$ of mutually disjoint $\n$-columns such that 
\[ \sum_{\C \in \mathfrak{C}} |I_{\top{\C}}| \lesssim \gamma^{r'n} = 2^{r_0 n}. \]
\end{lemma}
We have chosen to introduce the constant $\gamma$ in order to make the global decomposition lemma below more readily apparent. The proof of the decomposition lemma for $\mathbf{h}$ is quite similar to the one for $f$ and is thus omitted.\\
Finally, we can combine the two decomposition lemmas above iteratively in order to produce a global decomposition of the collection $\P$ as follows. 
%
% GLOBAL DECOMPOSITION LEMMA
%
\begin{lemma}\label{lemma_global_decomposition}
Let $\P$ be a collection of super tiles. Then there exists a decomposition $\P = \bigsqcup_{n} \P_n$ with the properties:
\begin{enumerate}[i)]
\item $\sizen_f(\P_n) \lesssim \min(2^{-n} \energyn_f(\P), \sizen_f(\P))$,
\item $\sizen_{\mathbf{h}}(\P_n) \lesssim \min(2^{-r_0 n/{r'}} \energyn_{\mathbf{h}}(\P), \sizen_{\mathbf{h}}(\P))$,
\item $\P_n$ is organized into a collection $\mathfrak{C}_n$ of mutually disjoint $\n$-columns,
\item $\sum_{\mathcal{C}\in\mathfrak{C}_n}|I_{\mathcal{C}}| \lesssim 2^{r_0 n}$.
\end{enumerate}
Furthermore, the collection $\P_n$ is empty if $n$ is such that 
\[ 2^{-n} \gtrsim \frac{\sizen_f(\P)}{\energyn_f(\P)} \quad \text{and} \quad 2^{-r_0 n/{r'}} \gtrsim \frac{\sizen_{\mathbf{h}}(\P)}{\energyn_{\mathbf{h}}(\P)}. \]
\end{lemma}
\begin{proof}
Initialize $\P_{\mathrm{stock}} := \P$ and apply iteratively the decomposition Lemmas \ref{lemma_decomposition_f} and \ref{lemma_decomposition_h}, in the order given by whichever of the quantities 
\[\frac{\sizen_f(\P_{\mathrm{stock}})}{\energyn_f(\P)}, \quad
\Big(\frac{\sizen_{\mathbf{h}}(\P_{\mathrm{stock}})}{\energyn_{\mathbf{h}}(\P)}\Big)^{{r'}/2} \]
is largest, putting the resulting $\P_{\mathrm{high}}$ collection in the corresponding $\P_n$ and updating $\P_{\mathrm{stock}}$ at the end of each step to be the collection $(\P_{\mathrm{stock}})_{\mathrm{low}}$ resulting from the last application of a decomposition lemma. We omit the details.
\end{proof}
%
%
%
%%%%%%%%%%%%%%%%
%
% GENERIC ESTIMATE 
%
\section{Generic estimate}\label{section_generic_estimate}
In this section we will combine the estimate for a given $\n$-column in Proposition \ref{proposition_column_control} with the global decomposition obtained in Section \ref{section_decomposition_lemmas} in order to obtain a global estimate for the trilinear form $\Lambda^{\n}_{\P}$ in terms of the sizes and energies of the collection $\P$. More precisely, we obtain the following (recall that $r > r_0 > 2$).
%
% GENERIC ESTIMATE
%
\begin{proposition}[Generic estimate]\label{proposition_generic_estimate}
Let $\sigma = (r-r_0)/r > 0$. Let $\P$ be a collection of super tiles, and denote for shortness
\begin{align*}
&\sizen_f(\P) =: S_f, \qquad \energyn_f(\P)=:E_f, \\
&\sizen_{\mathbf{h}}(\P) =: S_{\mathbf{h}}, \qquad \energyn_{\mathbf{h}}(\P)=:E_{\mathbf{h}}.
\end{align*}
Then for every $0 \leq \theta_1, \theta_2 \leq 1$ such that $\theta_1 + \theta_2  = 1$ we have 
\begin{equation}\label{eqn:general_estimate}
\begin{aligned}
|\Lambda^{\n}_{\P}(f,g,\mathbf{h})| \lesssim &  \sup_{P \in \P} \Big[\frac{1}{|I^{\n}_P|} \int_{I^{\n}_P} (\VarC{r}g)^r \Big]^{1/r}  \\
& \qquad \times S_f^{\sigma \theta_1} E_f^{1 - \sigma \theta_1} S_{\mathbf{h}}^{r' \sigma \theta_2 / r_0}  E_{\mathbf{h}}^{1 - r' \sigma \theta_2 / r_0 }.
\end{aligned}
\end{equation}
\end{proposition}
\begin{proof}
Apply the global decomposition lemma (Lemma \ref{lemma_global_decomposition}) to the collection $\P$, which yields the subcollections $\P_n$ and their partitions $\mathfrak{C}_n$ into $\n$-columns. For each such subcollection we then have by Proposition \ref{proposition_column_control} and Lemma \ref{lemma_global_decomposition} that we can bound
\begin{align*}
|\Lambda^{\n}_{\P_n}(f,g,\mathbf{h})| & \leq \sum_{\C \in \mathfrak{C}_n} |\Lambda^{\n}_{\C}(f,g,\mathbf{h})|  \\
& \lesssim  \sum_{\C \in \mathfrak{C}_n}  \Big[\frac{1}{|I_{\top{\C}}|} \sum_{P \in \C} \int_{I^{\n}_P} |\pi_{R_2(P)}g|^r \Big]^{1/r}  \sizen_f(\C) \sizen_{\mathbf{h}}(\C) |I_{\top{\C}}| \\
& \lesssim \sum_{\C \in \mathfrak{C}_n}  \Big[\frac{1}{|I_{\top{\C}}|} \sum_{P \in \C} \int_{I^{\n}_P} |\pi_{R_2(P)}g|^r \Big]^{1/r}\\
& \qquad\qquad\qquad\qquad \times \min(2^{-n}E_f, S_f) \min(2^{-r_0 n/{r'}}E_{\mathbf{h}}, S_{\mathbf{h}}) |I_{\top{\C}}|.
\end{align*}
By Lemma \ref{lemma_g_control} the term $\frac{1}{|I_{\top{\C}}|} \sum_{P \in \C} \int_{I^{\n}_P} |\pi_{R_2(P)}g|^r$ can be replaced by 
\[ \sup_{P \in \P} \frac{1}{|I^{\n}_P|} \int_{I^{\n}_P} (\VarC{r}g)^r, \]
and therefore we obtain by property \textit{(iv)} of Lemma \ref{lemma_global_decomposition} that 
\begin{align*}
|\Lambda^{\n}_{\P_n}(f,g,\mathbf{h})| \lesssim & \sup_{P \in \P} \Big[\frac{1}{|I^{\n}_P|} \int_{I^{\n}_P} (\VarC{r}g)^r\Big]^{1/r}  \\
& \times \min(2^{-n}E_f, S_f) \min(2^{-r_0 n/{r'}}E_{\mathbf{h}}, S_{\mathbf{h}}) \; 2^{r_0 n}.
\end{align*}
It thus suffices to show that upon summing in $n$ we obtain a quantity that is bounded exactly by the right hand side of \eqref{eqn:general_estimate}. This requires a tedious but easy case by case analysis.\\
Assume then that 
\[ \frac{S_f}{E_f} < \Big(\frac{S_{\mathbf{h}}}{E_{\mathbf{h}}}\Big)^{{r'}/r_0}, \]
the other cases being similar. We have to distinguish two situations:
\begin{enumerate}[1)]
\item \textbf{case} $2^{-n}\leq \frac{S_f}{E_f} < \Big(\frac{S_{\mathbf{h}}}{E_{\mathbf{h}}}\Big)^{{r'}/2}$: in this case we have to bound the sum 
\begin{align*}
\sum_{n \;:\; 2^{-n}\leq S_f (E_f)^{-1}} & 2^{-n} E_f 2^{-r_0 n/{r'}} E_{\mathbf{h}} 2^{r_0 n} 
& = E_f E_{\mathbf{h}} \sum_{n \;:\; 2^{-n}\leq S_f (E_f)^{-1}} 2^{-n (1 - r_0/r)};
\end{align*}
since $1 - r_0 / r = \sigma$ this sum is readily evaluated to be bounded by 
\begin{align*}
E_f E_{\mathbf{h}} \Big(\frac{S_f}{E_f}\Big)^{\sigma} &\leq E_f E_{\mathbf{h}} \Big(\frac{S_f}{E_f}\Big)^{\sigma \theta_1 } \Big(\frac{S_{\mathbf{h}}}{E_{\mathbf{h}}}\Big)^{r' \sigma \theta_2 / r_0 } \\
& = S_f^{\sigma \theta_1} E_f^{1 - \sigma\theta_1} S_{\mathbf{h}}^{r' \sigma \theta_2 /r_0 }  E_{\mathbf{h}}^{1 - r'\sigma\theta_2/r_0 }
\end{align*}
as desired.
\item \textbf{case} $\frac{S_f}{E_f} < 2^{-n} \leq \Big(\frac{S_{\mathbf{h}}}{E_{\mathbf{h}}}\Big)^{{r'}/r_0}$: in this case we have 
\begin{align*}
\sum_{n \;:\; S_f (E_f)^{-1} < 2^{-n} \leq S_{\mathbf{h}}^{r'/r_0} E_{\mathbf{h}}^{- r'/r_0}} &   S_f 2^{-r_0 n/{r'}} E_{\mathbf{h}} 2^{r_0 n} \\
& = S_f E_{\mathbf{h}} \sum_{n \;:\; S_f (E_f)^{-1} < 2^{-n} \leq S_{\mathbf{h}}^{r'/r_0} E_{\mathbf{h}}^{- r'/r_0}}  2^{r_0 n / r} \\
& \lesssim S_f E_{\mathbf{h}}\Big(\frac{S_f}{E_f}\Big)^{- r_0 / r} = E_f E_{\mathbf{h}}\Big(\frac{S_f}{E_f}\Big)^{ \sigma},
\end{align*}
which we have already seen is acceptable in the previous case.
\end{enumerate}
This concludes the proof.
\end{proof}
%
%
%%%%%%%%%%%%%%%%
%
% proof in the limited range when p=r or q=r
%
\section{Preliminary result when $p=r$ or $q = r$}\label{section:partial_results}
In this section we will prove a preliminary result that, when combined with the interpolation technique in Section \S \ref{section:proof_of_main_theorem} will yield Theorem \ref{main_theorem_2}. In particular, we will prove restricted weak type estimates for our trilinear forms with either $p = r$ (for rectangles of low eccentricity) or $q = r$ (for rectangles of high eccentricity).\par
We begin by stating rigorously the results that will be proven in this section.  There are two such results, one for each of the collections $\mathscr{R}_{\mathrm{high}}$, $\mathscr{R}_{\mathrm{low}}$ defined in Section \S\ref{section_reductions}.\\
Let $A$ be a measurable set. We say that a measurable set $A'$ is a \emph{major subset} of $A$ if $A' \subset A$ and $|A'| > \frac{1}{2} |A|$.
%
% PARTIAL RESULT FOR r>2, WHERE q = r
%
\begin{proposition}[Rectangles of high eccentricity]\label{proposition_partial_result_r>2_ecc_large}
Let $r>r_0>2$ and let $\mathscr{R}_{\mathrm{high}}$ be a collection of disjoint rectangles such that for every $R \in \mathscr{R}_{\mathrm{high}}$ it holds that $\ec{R} > 1$.
Let $p$ be such that 
\[ \frac{1}{100}\frac{1}{r_0} + \frac{99}{100}\frac{1}{r} \leq \frac{1}{p} \leq \frac{1}{r_0} \]
and $s$ be such that 
\[ \frac{1}{p} + \frac{1}{r} = \frac{1}{s}. \]
Then, for any $F,G,H$ finite measurable subsets of $\R$ and for any measurable functions $f,g$ such that 
\[ |f|\leq \mathds{1}_F, \quad |g|\leq \mathds{1}_{G}, \]
there exists a major subset $H' \subset H$ such that for any measurable vector-valued function $\mathbf{h} = (h_R)_{R\in\mathscr{R}_{\mathrm{high}}}$ that satisfies
\[ \Big(\sum_{R \in\mathscr{R}_{\mathrm{high}}}|h_R|^{r'}\Big)^{1/{r'}} \leq \mathds{1}_{H'} \]
it holds that 
\[ |\Lambda^r_{\mathscr{R}_{\mathrm{high}}}(f,g,\mathbf{h})| \lesssim_{p,r} |F|^{1/p} |G|^{1/r}|H|^{1/{s'}}. \]
The choice of major subset may depend on $r,r_0, f,g, H$ but not on $p,s$.
\end{proposition}
Observe that the range of exponents is restricted to a subset of the line $q=r$.
\begin{remark}\label{remark:partial_result_statement}
The rather clumsy appearance of the above statement is due to some technicalities. Some comments are in order.\\ Firstly, there is the fact that in our argument the major subset $H'$ will end up depending on $f,g$ rather than on the sets $F,G$ only - the latter being the standard in typical time-frequency analysis arguments. This is because we will control the contributions of $f$ and $g$ by non-local non-positive averages (see Lemma \ref{lemma_f_control} and Lemma \ref{lemma_g_control}) and thus we will not be allowed to replace $f$ and $g$ by $\mathds{1}_F, \mathds{1}_G$ in the definition of the exceptional set that we are going to remove from $H$ (see proof below). However, this will not be a problem for the interpolation argument in the quasi-Banach range (that is, when $s < 1$), as will be clear from Lemma \ref{lemma_interpolation} and its proof.\\
Secondly, the range of boundedness in the statement looks unnatural and arbitrary. This is due to the fact that, in order to perform interpolation between estimates for different values of $r$, we need the major subset to be independent of $p,s$. If we gave up this requirement, the methods employed in the proof would yield the more natural range $2 < p < r$, but then we would not be able to ensure that the choice of major subset is indeed independent of $p$. In view of a more transparent proof, we have preferred the above formulation.
\end{remark}
 Similarly, we have
\begin{proposition}[Rectangles of low eccentricity]\label{proposition_partial_result_r>2_ecc_small}
Let $r>r_0 > 2$ and let $\mathscr{R}_{\mathrm{low}}$ be a collection of disjoint rectangles such that for every $R \in \mathscr{R}_{\mathrm{low}}$ it holds $\ec{R} < 1$.
Let $q$ be such that 
\[ \frac{1}{100}\frac{1}{r_0} + \frac{99}{100}\frac{1}{r} \leq \frac{1}{q} \leq \frac{1}{r_0} \]
and $s$ be such that 
\[ \frac{1}{r} + \frac{1}{q} = \frac{1}{s}. \]
Then, for any $F,G,H$ finite measurable subsets of $\R$ and for any functions $f,g$ such that 
\[ |f|\leq \mathds{1}_F, \quad |g|\leq \mathds{1}_{G}, \]
there exists a major subset $H' \subset H$ such that for any measurable vector-valued function $\mathbf{h}= (h_R)_{R\in\mathscr{R}_{\mathrm{low}}}$ that satisfies
\[ \Big(\sum_{R\in\mathscr{R}_{\mathrm{low}}}|h_R|^{r'}\Big)^{1/{r'}} \leq \mathds{1}_{H'} \]
it holds that 
\[ |\Lambda^r_{\mathscr{R}_{\mathrm{low}}}(f,g,\mathbf{h})| \lesssim_{q,r} |F|^{1/r} |G|^{1/q}|H|^{1/{s'}}. \]
The choice of major subset $H'$ may depend on $r,r_0,f,g,H$ but not on $q,s$.
\end{proposition}
The range of exponents this time is restricted to a subset of the line $p=r$ instead. The same remarks as in \ref{remark:partial_result_statement} apply. \\
Although the ranges in the two above propositions don't intersect, upon interpolation with the $r=\infty$ case they will each yield the same range of estimates, allowing us to sum their two contributions. See next section for details.\\
We now proceed with the proof. We will prove the first of the two propositions above and then comment in Remark \ref{remark_on_trilinear_form_eccentricity_ll_1} on the modifications one has to make to the argument in order to adapt it to the second one. 
\begin{proof}[Proof of Prop. \ref{proposition_partial_result_r>2_ecc_large}]
%
% proof following the usual Muscalu, Tao, Thiele argument
%
The proof follows a standard argument originating from \cite{MuscaluTaoThiele} (although implicitly present in previous work), together with ideas from \cite{Benea_thesis,BeneaMuscalu}, in particular the idea of using non-local operators for the stopping-time argument below. \\
It will suffice to prove, for data as given in the statement, that for any collection of super tiles $\P$ it is 
\begin{equation}\label{eqn:restricted_weak_type_estimate}
 \Lambda^{\n}_{\P}(f,g,\mathbf{h}) \lesssim_{p,q,r} [\log^{+}(\n)]^{O(1)} |F|^{1/p} |G|^{1/r} |H|^{1/{s'}} 
 \end{equation}
for $\frac{1}{100 r_0} + \frac{99}{100 r} \leq \frac{1}{p} \leq \frac{1}{r_0}$ and $1/p + 1/r + 1/{s'} = 1$. Indeed, since the constant is only poly-logarithmic in the shifting parameter $\n$ and since by Section \S\ref{section_reductions} the trilinear form $\Lambda^r_{\mathscr{R}_{\mathrm{high}}}$ is dominated by 
\[ \sum_{\n \in \Z} (1 + |\n|)^{-N} \Lambda^{\n}_{\P}(f,g,\mathbf{h}) \]
for an arbitrarily large $N>0$, the above implies the same bounds for $\Lambda^r_{\mathscr{R}_{\mathrm{high}}}$ itself.\\
%
% definition of Exceptional Set E
%
Given sets $F,G,H$ and functions $f,g$ such that 
\[ |f| \leq \mathds{1}_F, \quad |g| \leq \mathds{1}_G, \]
we set out to define an exceptional set $E$ (independent of $p,s$ but depending on $f,g$) which will be removed from $H$. It will be useful to introduce the following maximal operators: for any $\m \in \Z$, define $\mathcal{M}^{\m}$ to be the shifted maximal function
\[ \mathcal{M}^{\m}f(x):= \sup_{\substack{I \text{ dyadic}, \\ x \in I }} \frac{1}{|I|} \int_{I^{\m}} |f(y)| \ud y. \]
It is well known that the shifted maximal function $\mathcal{M}^{\m}$ is $L^p \to L^p$ bounded for $1 < p \leq \infty$ and also $L^1 \to L^{1,\infty}$ bounded, with constant at most $O(\log^{+}(\m))$ (see for example \S 4.3.4 of \cite{MuscaluSchlag}). Now we can define the exceptional set $E$ to be  
\begin{align*}
E:=  \bigcup_{\m \in \Z} & \Big\{x \in \R \text{ s.t. } \mathcal{M}^{-\n + \m}((\mathscr{C}f)^{100 r_0})(x) \gtrsim \log^{+}(\m - \n) \langle \m \rangle^2 \frac{|F|}{|H|} \Big\} \\
& \cup \bigcup_{\m' \in \Z} \Big\{x \in \R \text{ s.t. } \mathcal{M}^{\m'}((\VarC{r}g)^r)(x) \gtrsim \log^{+}(\m') \langle \m' \rangle^2 \frac{|G|}{|H|}\Big\}
\end{align*}
with implicit constants to be chosen below, where $\langle \m \rangle := (1 + |\m|^2)^{1/2}$. 
%
% proving the exceptional set is small compared to H
%
Define then the set $H':= H \backslash E$; we claim that $H'$ is a major subset of $H$. Indeed, for a given $\m \in \Z$ we see by the $L^1 \to L^{1,\infty}$ boundedness of $\mathcal{M}^{-\n + \m}$ that 
\begin{align*}
\Big|\Big\{x \in \R \text{ s.t. } & \mathcal{M}^{-\n + \m}((\mathscr{C}f)^{100 r_0})(x) \gtrsim \log^{+}(\m - \n) \langle \m \rangle^2 \frac{|F|}{|H|} \Big\}\Big| \\
&  \lesssim \log^{+}(\m - \n) \frac{\int |\mathscr{C}f|^{100 r_0} \ud x}{\log^{+}(\m - \n) \langle \m \rangle^2 \frac{|F|}{|H|}} \\
& \lesssim \langle \m \rangle^{-2}|H| \frac{\int |f|^{100 r_0}}{|F|} \leq 
 \langle \m \rangle^{-2} |H|, 
\end{align*}
where in the last line we have used the fact that the Carleson operator $\mathscr{C}$ is $L^{100 r_0} \to L^{100 r_0}$ bounded. Since this is summable in $\m$, we see that for a suitable choice of implicit constants the contribution to $|E|$ of these sets is at most $\ll |H|$. Similarly, for a given $\m' \in \Z$ we see by the $L^1 \to L^{1,\infty}$ boundedness of $\mathcal{M}^{\m'}$ and by the $L^r \to L^r$ boundedness of $\VarC{r}g$ that 
\[ \Big|\Big\{x \in \R \text{ s.t. } \mathcal{M}^{\m'}((\VarC{r}g)^r)(x) \gtrsim \log^{+}(\m') \langle \m' \rangle^2 \frac{|G|}{|H|} \Big\}\Big| \ll \langle \m' \rangle^{-2} |H|, \] 
which is again summable in $\m'$ and thus we conclude that $|E| \ll |H|$, which proves the claim. Notice $H'$ depends on $f,g,r,r_0$ and $H$ but not on $p$.\\
%
% splitting the partition according to whether I_P is contained in the exceptional set or not
%
Now, we partition the collection $\P$ into
\begin{align*}
\P_{\mathrm{small}} :=& \{P \in \P \text{ s.t. } I^{\n}_P \not\subset E\}, \\
\P_{\mathrm{large}}:=& \P \backslash \P_{\mathrm{small}},
\end{align*}
and will estimate the trilinear forms $\Lambda^{\n}_{\P_{\mathrm{small}}}$ and $\Lambda^{\n}_{\P_{\mathrm{large}}}$ separately. Notice the presence of the shifting parameter $\n$ in the definitions.\\
%
% proof for shifted super tiles that are not contained in E
%
% estimates for every term that appears in the general estimate
%
We start with $\P_{\mathrm{small}}$. Given $P \in \P_{\mathrm{small}}$ we observe that since $I^{\n}_P \not\subset E$ there exists an $x \in I^{\n}_P$ such that $\mathcal{M}^{-\n}((\mathscr{C}f)^{100 r_0})(x) \lesssim \log^{+}(\n)|F| / |H|$, and thus a fortiori it must be that 
\[ \frac{1}{|I_P|} \int_{(I^{\n}_P)^{-\n}} |\mathscr{C}f|^{100 r_0} \ud x \lesssim \log^{+}(\n)\frac{|F|}{|H|}; \]
but $(I^{\n}_P)^{-\n} = I_P$, and therefore we must have by Lemma \ref{lemma_f_control} that
\begin{equation}\label{eqn:estimate_size_f}
\sizen_f(\P_{\mathrm{small}}) \lesssim \log^{+}(\n)\Big(\frac{|F|}{|H|}\Big)^{1/ {100 r_0}} 
\end{equation}
(we are not keeping track of the optimal powers of $\log^{+}(\n)$ because they will be inconsequential to us).
Similarly, we see that for a given $P \in \P_{\mathrm{small}}$ there must be an $x \in I^{\n}_P$ such that $\mathcal{M}^{0}((\VarC{r}g)^r)(x) \lesssim |G| / |H|$, and therefore for all such $P$ we have
\begin{equation}\label{eqn:estimate_g}
\frac{1}{|I_P|} \int_{I^{\n}_P}(\VarC{r}g)^r \ud x \lesssim \frac{|G|}{|H|}.
\end{equation}
Moreover, by Lemma \ref{lemma_h_control} and the hypothesis that $\|\mathbf{h}\|_{\ell^{r'}} \leq \mathds{1}_{H'}$ we have the trivial bound 
\begin{equation}\label{eqn:estimate_size_h}
\sizen_{\mathbf{h}}(\P) \lesssim 1 
\end{equation}
for any arbitrary collection of super tiles $\P$.\\
%
% applying the estimates to the general estimate putting everything together
%
Recall that $\sigma = (r - r_0)/r$. Combining \eqref{eqn:estimate_size_f}, \eqref{eqn:estimate_g} and \eqref{eqn:estimate_size_h} with the generic estimate of Proposition \ref{proposition_generic_estimate} and the energy estimates of Lemmas \ref{lemma_energy_f}, \ref{lemma_energy_h}, we obtain after a little algebra that for any $0 \leq \theta_1, \theta_2 \leq 1$, $\theta_1 + \theta_2 = 1$, 
\begin{equation}\label{eqn:restricted_weak_type_estimate_2}
\begin{aligned}
|\Lambda^{\n}_{\P_{\mathrm{small}}}(f,g,\mathbf{h})|  \lesssim & \log^{+}(\n)^{O(1)} \Big(\frac{|G|}{|H|}  \Big)^{1/r}  \Big(\frac{|F|}{|H|}\Big)^{\sigma \theta_1 / 100 r_0} \\
& \qquad \qquad \times (|F|^{1/r_0})^{1 - \sigma\theta_1} \cdot 1 \cdot |H|^{1/{r'} - \sigma\theta_2/r_0 } \\
& =  \log^{+}(\n)^{O(1)} |F|^{1/r_0 - (99/100)\sigma\theta_1/r_0 } |G|^{1/r}\\
& \qquad \qquad \times |H|^{1/{r'} - 1/r - \sigma\theta_2/r_0 - \sigma \theta_1 / 100 r_0}.
\end{aligned}
\end{equation}
Since 
\[ \frac{1}{r_0} - \frac{99}{100} \frac{\sigma}{r_0} = \frac{1}{100 r_0} + \frac{99}{100 r_0}, \]
this yields precisely the desired estimates \eqref{eqn:restricted_weak_type_estimate} as $\theta_1$ ranges over $[0,1]$.\\
%
% proof for the tiles that are contained in E (essentially a repetition of the above)
%
Now we are left with showing that \eqref{eqn:restricted_weak_type_estimate} holds for $\Lambda^{\n}_{\P_{\mathrm{large}}}$ as well. In order to do so, we decompose $\P_{\mathrm{large}}$ into $\bigsqcup_{d \in \N} \P_{d}$ where
%
% decomposition according to the distance from the complement of E
%
\[ \P_d := \Big\{P \in \P_{\mathrm{large}} \text{ s.t. } 2^{d}\leq 1 + \frac{\mathrm{dist}(I^{\n}_P, E^c)}{|I_P|} < 2^{d+1} \Big\}; \]
it then suffices to prove that the contribution of $\Lambda^{\n}_{\P_d}$ is summable in $d$ and that the sum is bounded by the right hand side of \eqref{eqn:restricted_weak_type_estimate_2}. \\
%
% Estimates for every term that appears in the general estimate for \P_d
%
Let then $d$ be fixed and observe that if $P \in \P_d$ then $2^{d + 1} I^{\n}_P \not\subset E$ and $2^{d} I^{\n}_P \subset E$. This means that for some integer $\mathfrak{l}$ such that $|\mathfrak{l}| \sim 2^d$ the shifted interval $(I^{\n}_P)^{\mathfrak{l}} = I^{\n + \mathfrak{l}}_P$ is not contained in $E$. This implies, by definition of $E$, that for some $x \in I^{\n + \mathfrak{l}}_P$ we have $\mathcal{M}^{-\n -\mathfrak{l}}(\mathscr{C}f^{100 r_0})(x) \lesssim \log^{+}(\n + \mathfrak{l}) \langle \mathfrak{l} \rangle^2 {|F|/|H|}$, which in turn implies as before 
\begin{equation}\label{eqn:estimate_size_f_2}
\sizen_f(\P_d) \lesssim (\sup_{|\mathfrak{l}|\sim 2^d}\log^{+}(\n + \mathfrak{l})) 2^{2d / 100 r_0} \Big(\frac{|F|}{|H|}\Big)^{100 r_0}.
\end{equation}
Similarly, given $P \in \P_d$ we see that there must be an integer $\mathfrak{l}$ such that $|\mathfrak{l}| \sim 2^d$ and the shifted interval $(I^{\n}_P)^{\mathfrak{l}} = I^{\n + \mathfrak{l}}_P$ contains a point $x$ such that $\mathcal{M}^{-\mathfrak{l}}((\VarC{r}g)^r)(x) \lesssim \log^{+}(\mathfrak{l}) \langle \mathfrak{l} \rangle^2 {|G|/|H|}$; this in turn implies the estimate
\begin{equation}\label{eqn:estimate_g_2}
\sup_{P \in \P_d} \frac{1}{|I_P|} \int_{I^{\n}_P} (\VarC{r}g)^r \ud x \lesssim d 2^{2d} \frac{|G|}{|H|}.
\end{equation}
Now, estimates \eqref{eqn:estimate_size_f_2}, \eqref{eqn:estimate_g_2} are worse than the corresponding ones \eqref{eqn:estimate_size_f}, \eqref{eqn:estimate_g}; however, we now have an improved estimate for $\mathbf{h}$. Indeed, if $P \in \P_d$, we must have by definition of $H'$ that $\mathrm{dist}(I^{\n}_P, H') \gtrsim 2^d |I_P|$; then, by the rapid decay of the functions $\Phi_{I^{\n}_P}$, for any large $M>0$ we can bound (again by Lemma \ref{lemma_h_control}) 
\begin{equation}\label{eqn:estimate_size_h_2}
\sizen_{\mathbf{h}}(\P_d) \lesssim_M 2^{-d M}.
\end{equation}
Combining estimates \eqref{eqn:estimate_size_f_2}, \eqref{eqn:estimate_g_2} and \eqref{eqn:estimate_size_h_2} with Proposition \ref{proposition_generic_estimate} and the energy Lemmas \ref{lemma_energy_f}, \ref{lemma_energy_h} as done before, we obtain for any $0 \leq \theta_1, \theta_2 \leq 1$, $\theta_1 + \theta_2 = 1$, that the estimate
\begin{align*}
 |\Lambda^{\n}_{\P_d}(f,g,\mathbf{h})| \lesssim (\sup_{|\mathfrak{l}|\sim 2^d}\log^{+}(\n + \mathfrak{l}))^{O(1)} &  d 2^{-d M'}  |F|^{1/r_0 - (99/100)\sigma\theta_1/r_0 } \\
 & \times |G|^{1/r} |H|^{1/{r'} - 1/r - \sigma\theta_2/r_0 - \sigma \theta_1 / 100 r_0} 
\end{align*}
holds, where $M' > M/2$, provided $M$ was chosen sufficiently large. This is summable in $d > 0$, and since 
\[ \sum_{d > 0} \sum_{\mathfrak{l} \; : \; |\mathfrak{l}|\sim 2^d} \log^{+}(\n + \mathfrak{l})^{O(1)}  d 2^{-d M'} \lesssim (\log^{+}(\n))^{O(1)}  \]
we see that $|\Lambda^{\n}_{\P_{\mathrm{large}}}(f,g,\mathbf{h})|$ is controlled by the right hand side of \eqref{eqn:restricted_weak_type_estimate_2} as desired, thus concluding the proof.
\end{proof}
%
% REMARK ON A PROOF FOR PROPOSITION \ref{proposition_partial_result_r>2_ecc_small}
%
\begin{remark}\label{remark_on_trilinear_form_eccentricity_ll_1}
We conclude the section with a comment on how Proposition \ref{proposition_partial_result_r>2_ecc_small} is proven. In the previous sections we have concentrated on the collection $\mathscr{R}_{\mathrm{high}}$ of rectangles with large eccentricity (i.e. $|R_2| > |R_1|$). However, it is clear that the analysis is completely symmetric in the opposite case of $\mathscr{R}_{\mathrm{low}}$, in which the r\^{o}les of functions $f$ and $g$ are simply swapped. In particular, one has to re-define super tiles reversing $f$ and $g$ and use $\n$-rows instead of $\n$-columns (adapting the definitions is a trivial exercise); then one performs a global decomposition of the collection of super tiles into uniformly controlled $\n$-rows. All lemmas and propositions from previous sections then hold with $f$ replaced by $g$ and viceversa, without substantial changes.
\end{remark}
%
%
%%%%%%%%%%%%
%
% COMPLETING THE PROOF
%
\section{Proof of the main theorem}\label{section:proof_of_main_theorem}
In this section we will finally complete the proof of Theorem \ref{main_theorem_2}. This will be achieved by interpolating the results of Section \S \ref{section:partial_results} for generic $r>2$ with the trivial case of $r=\infty$, as in the authors' related work \cite{BernicotVitturi}.\par
First of all, we state (and prove) the $r=\infty$ case of Theorem \ref{main_theorem_2}. 
%
% CASE r = \infty
%
\begin{theorem}\label{theorem_r_infty}
Let $\mathscr{R}$ be a collection of disjoint rectangles. Then for all $1 < p, q < \infty$ and 
\[ \frac{1}{p} + \frac{1}{q} = \frac{1}{s} \]
it holds that for any $f \in L^p, g \in L^q$
\[ \|S^{\infty}_{\mathscr{R}}(f,g)\|_{L^s} \lesssim_{p,q} \|f\|_{L^p} \|g\|_{L^q}.  \]
\end{theorem}
\begin{proof}
Notice that 
\[ S^{\infty}_{\mathscr{R}}(f,g)(x) = \sup_{R \in \mathscr{R}} |\pi_{R_1}f(x) \cdot \pi_{R_2}g(x)| \]
is dominated pointwise by $\mathscr{C}f \cdot \mathscr{C}g$. The result is then a trivial consequence of H\"{o}lder's inequality and the Carleson-Hunt theorem.
\end{proof}
\begin{remark}\label{remark_ell_1_trilinear_form}
We record here for future reference that Theorem \ref{theorem_r_infty} implies the $L^p \times L^r \times L^{s'}(\ell^1) \to \R$ boundedness of the trilinear form $\Lambda^{\infty}_{\mathscr{R}}(f,g,\mathbf{h})$ as given in \eqref{eqn:trilinear_form}.
\end{remark}
In order to combine this trivial result with Propositions \ref{proposition_partial_result_r>2_ecc_large}, \ref{proposition_partial_result_r>2_ecc_small}, we will need a multilinear interpolation argument for vector-valued operators. This argument is originally due to Silva, who introduced it in \cite{Silva}, and is based upon complex interpolation. We introduce the following definition that extends the usual notion of generalized restricted weak type inequalities to the vector-valued setting we need.
%
% definition of generalized restricted weak type with vector valued operators
%
\begin{definition}\label{definition_generalized_restricted_weak_type}
Let $\Lambda$ be a trilinear form, let ${\boldsymbol\alpha}= (\alpha_1,\alpha_2,\alpha_3)$ be such that $0 \leq \alpha_1, \alpha_2 \leq 1$, $\alpha_3\leq 1$, $\alpha_1 + \alpha_2 + \alpha_3 = 1$ and let $t\geq 1$. We say that $\Lambda$ is of \emph{generalized restricted weak type} $({\boldsymbol\alpha};t)$ if for every measurable subsets $F,G,H$ of $\R$ of finite measure and for every functions $f,g$ such that 
\[ |f|\leq \mathds{1}_F, \quad |g|\leq \mathds{1}_G, \]
there exists a major subset $H'$ of $H$ such that for any vector-valued measurable function $\mathbf{h} = (h_k)_k$ that satisfies
\[ \Big(\sum_{k}|h_k|^t\Big)^{1/t} \leq \mathds{1}_{H'}, \] 
the inequality 
\[ |\Lambda(f,g,\mathbf{h})| \lesssim |F|^{\alpha_1}|G|^{\alpha_2}|H|^{\alpha_3}\]
holds true.
\end{definition}
\begin{remark}
The difference between Definition \ref{definition_generalized_restricted_weak_type} and the classical definition of generalized restricted weak type (see for example Chapter 3 in \cite{Thiele_book}) is two-fold. Firstly, they differ in the presence of the additional parameter $t$ which specifies the space where the vector-valued function $\mathbf{h}$ takes values; secondly, the major subset $H'$ is here allowed to depend on $f,g$ instead of $F,G$ only. 
\end{remark}
%
% remark on the result of the previous section phrased in this new terminology
%
\begin{remark}\label{remark:preliminary_result_using_generalized_restricted_weak_type}
Using Definition \ref{definition_generalized_restricted_weak_type} we can rephrase Proposition \ref{proposition_partial_result_r>2_ecc_large} as stating that, under the same hypotheses, the trilinear form $\Lambda^r_{\mathscr{R}_{\mathrm{high}}}$ is of generalized restricted weak type
\[ \Big(\Big(\frac{1}{p}, \frac{1}{r}, \frac{1}{s'}\Big); r'\Big)  \]
for all $p,s$ such that $\frac{1}{100 r_0} + \frac{99}{100 r} \leq \frac{1}{p} \leq \frac{1}{r_0}$ and $\frac{1}{p} + \frac{1}{r} + \frac{1}{s'} = 1$.\\
Similarly, Proposition \ref{proposition_partial_result_r>2_ecc_small} states that the trilinear form $\Lambda^r_{\mathscr{R}_{\mathrm{low}}}$ is of generalized restricted weak type
\[ \Big(\Big(\frac{1}{r}, \frac{1}{q}, \frac{1}{s'}\Big); r'\Big)  \]
for all $q$ such that $\frac{1}{100 r_0} + \frac{99}{100 r} \leq \frac{1}{q} \leq \frac{1}{r_0}$ and $\frac{1}{r} + \frac{1}{q} + \frac{1}{s'} = 1$.\\
Finally, Theorem \ref{theorem_r_infty} above implies that $\Lambda^\infty_{\mathscr{R}}$ is of generalized restricted weak type 
\[ \Big(\Big(\frac{1}{p}, \frac{1}{q}, \frac{1}{s'}\Big); 1\Big)  \]
for all $1 < p,q < \infty$ and $\frac{1}{p} + \frac{1}{q} + \frac{1}{s'} = 1$.
\end{remark}
The interpolation argument can now be stated as follows.
%
% MULTILINEAR INTERPOLATION OF VECTOR VALUED OPERATORS
%
\begin{lemma}[\cite{Silva}]\label{lemma_interpolation}
Let $\Lambda$ be a trilinear form of generalized restricted weak type $({\boldsymbol\alpha};t_0)$ and $({\boldsymbol\beta};t_1)$, with the property that the major subset is the same for $({\boldsymbol\alpha};t_0)$ and $({\boldsymbol\beta};t_1)$. Then for all $\theta$ such that $0 < \theta < 1$, with ${\boldsymbol\alpha}^{\theta}$ given by
\[ \alpha^{\theta}_j = (1-\theta)\alpha_j + \theta \beta_j, \qquad j=1,2,3  \]
and $t_\theta$ given by
\[ \frac{1}{t_\theta} = \frac{1-\theta}{t_0} + \frac{\theta}{t_1}, \]
it holds that $\Lambda$ is of generalized restricted weak type $({\boldsymbol\alpha}^{\theta}; t_\theta)$.
\end{lemma}
One can also bound the constant of the interpolated inequality explicitely in terms of the data, but we will not be interested in doing so here.
\begin{proof}
The lemma is a particular case of a more general interpolation lemma of Silva (specifically Lemma 4.3 in \cite{Silva}). We sketch the proof here for the reader's convenience.\\
We argue by complex interpolation and assume $t_0, t_1 < \infty$. Let $F,G,H,f,g,H',\theta$ be given and let $\mathbf{h}$ be such that 
\[ \Big(\sum_k |h_k|^{t_\theta}\Big)^{1/{t_\theta}} \leq \mathds{1}_{H'}.\] 
For $z\in \mathbb{C}$ with $\mathrm{Re} z \in [0,1]$ define $\mathbf{h}^z$ by 
\[ h_k^z(x) := |h_k(x)|^{t(z)} \]
for every $k$, where 
\[ t(z) := (1-z)\frac{t_\theta}{r_0} + z \frac{t_\theta}{r_1}. \]
When $\mathrm{Re} z = 0$ we have $|h_k^z|^{t_0} = |h_k|^{t_\theta}$, and when $\mathrm{Re} z = 1$ we have $|h_k^z|^{t_1} = |h_k|^{t_\theta}$; hence by assumption we have for $\mathrm{Re} z = 0$
\[ |\Lambda(f,g,\mathbf{h}^z)| \lesssim |F|^{\alpha_1}|G|^{\alpha_2}|H|^{\alpha_3}, \]
and for $\mathrm{Re} z = 1$ we have
\[ |\Lambda(f,g,\mathbf{h}^z)| \lesssim |F|^{\beta_1}|G|^{\beta_2}|H|^{\beta_3}. \]
Since the function $\Psi(z) := \Lambda(f,g,\mathbf{h}^z)$ is easily seen to be holomorphic in the open strip $S = \{z \in \mathbb{C} \text{ s.t. } 0 < \mathrm{Re} z < 1 \}$, continuous in its closure and bounded, we can apply to it Hadamard's three-lines-lemma and conclude that since $\mathbf{h}^{\theta + i0} = \mathbf{h}$ we have
\[ |\Lambda(f,g,\mathbf{h})| \lesssim |F|^{\alpha^{\theta}_1}|G|^{\alpha^{\theta}_2}|H|^{\alpha^{\theta}_3}, \]
as desired.
\end{proof}
We are now ready to prove Theorem \ref{main_theorem_2}. It will be a straightforward consequence of the interpolation Lemma \ref{lemma_interpolation}, but some extra care is needed to make sure all hypotheses are verified. The resulting proof is elementary but technical and painful to read; in order to remedy this fact, we have included in Figure \ref{figure_interpolation} a picture that illustrates the proof geometrically.

\begin{figure}[ht]
\centering
\definecolor{zzttqq}{rgb}{0.6,0.2,0}
\definecolor{ffxfqq}{rgb}{1,0.5,0}
\definecolor{sqsqsq}{rgb}{0.125,0.125,0.125}
\definecolor{uuuuuu}{rgb}{0.267,0.267,0.267}
\begin{tikzpicture}[line cap=round,line join=round,>=triangle 45,x=1cm,y=1cm]
\clip(-0.57,-0.6) rectangle (16.0,9);
%
%%%%%%%%
%
% RANGES
%
% total range (r',r) at r_2 
\fill[line width=2pt,color=ffxfqq,fill=ffxfqq,fill opacity=0.18] (3.114202852374399,1.311243306262905) -- (5.13077037267722,3.3278108265657265) -- (5.130770372677221,6.688756693737095) -- (3.114202852374399,4.672189173434274) -- cycle;
% partial range at r_2
\fill[line width=2pt,color=zzttqq,fill=zzttqq,fill opacity=0.35] (3.114202852374399,3.9996337876048416) -- (4.485117202280964,5.370548137511407) -- (4.485117202280964,2.816668733335356) -- (3.114202852374399,1.4457543834287914) -- cycle;
% total range (r',r) at r_1 
\fill[line width=2pt,color=zzttqq,fill=zzttqq,fill opacity=0.10] (5.734878676629043,2.4146857585806494) -- (6.923864357693555,3.603671439645162) -- (6.923864357693555,5.585314241419351) -- (5.734878676629043,4.396328560354837) -- cycle;
% total range at r=\infty
\fill[line width=2pt,color=zzttqq,pattern=north west lines,pattern color=zzttqq,fill opacity=0.40] (0,5) -- (3,8) -- (3,3) -- (0,0) -- cycle;
%
%%%%%%%
%
% SEGMENTS
%
% segments defining the parallelepiped 1/p \times 1/q \times r \in [0,1]\times[0,1]\times[2,\infty]
\draw [line width=1pt] (0,0)-- (3,3);
\draw [line width=1pt] (8,0)-- (0,0);
\draw [line width=1pt] (8,5)-- (0,5);
\draw [line width=1pt] (8,0)-- (8,5);
\draw [line width=1pt] (11,3)-- (3,3);
\draw [line width=1pt] (8,0)-- (11,3);
\draw [line width=1pt] (0,5)-- (0,0);
\draw [line width=1pt] (3,3)-- (3,8);
\draw [line width=1pt] (3,8)-- (0,5);
\draw [line width=1pt] (11,3)-- (11,8);
\draw [line width=1pt] (11,8)-- (8,5);
\draw [line width=1pt] (11,8)-- (3,8);
% central segment connecting all the L^2\times L^2 \to L^1 estimates
\draw [line width=0.8pt,dotted] (1.5,4)-- (9.5,4);
% segments connecting the L^2\times L^2 \to L^1 estimate to the corners of the r=\infty square
\draw [line width=0.8pt,dotted] (9.5,4)-- (0,0);
\draw [line width=0.8pt,dotted] (9.5,4)-- (3,3);
\draw [line width=0.8pt,dotted] (9.5,4)-- (3,8);
\draw [line width=0.8pt,dotted] (9.5,4)-- (0,5);
% segments connecting the (p=r, r_0 < q < ...) range to the corners of the r=\infty square
\draw [line width=0.8pt,dash pattern=on 2pt off 2pt] (5.734878676629043,2.6623911088024226)-- (0,0);
\draw [line width=0.8pt,dash pattern=on 2pt off 2pt] (5.734878676629043,3.15780180924597)-- (0,5);
\draw [line width=0.8pt,dash pattern=on 2pt off 2pt] (5.734878676629043,3.15780180924597)-- (3,8);
\draw [line width=0.8pt,dash pattern=on 2pt off 2pt] (5.734878676629043,2.6623911088024226)-- (3,3);
% segment highlighting the (p=r, r_0 < q < ...) range
\draw [line width=2pt] (5.734878676629043,3.15780180924597)-- (5.734878676629043,2.6623911088024226);
% segment highlighting the (r_0 < p < ..., q=r) range
\draw [line width=2pt] (6.180748307028235,2.8605553889798414)-- (5.883501886762106,2.5633089687137134);
% segments giving the coordinate projections of the ranges at r_1 and r_2
\draw [line width=0.6pt,dotted] (2.62248661252581,0)-- (5.622486612525809,3);
\draw [line width=0.6pt,dotted] (4.829371517161299,0)-- (7.829371517161299,3);
\draw [line width=0.6pt,dotted] (5.13077037267722,3.3278108265657265)-- (5.130770372677221,2.5082837601514103);
\draw [line width=0.6pt,dotted] (6.923864357693555,3.603671439645162)-- (6.923864357693556,2.094492840532256);
\draw [line width=0.6pt,dotted] (3.114202852374399,1.311243306262905)-- (3.114202852374399,0.4917162398485895);
\draw [line width=0.6pt,dotted] (5.734878676629043,0.9055071594677438)-- (5.73487867662904,2.4146857585806485);
\draw [line width=0.6pt,dotted] (5.734878676629043,0.9055071594677438)-- (8.905507159467744,0.9055071594677438);
\draw [line width=0.6pt,dotted] (3.114202852374399,0.4917162398485895)-- (8.491716239848589,0.4917162398485895);
% segments giving the coordinate projections of the L^2 \times L^2 \to L^1 estimate at r=2
\draw [line width=0.6pt,dotted] (9.5,4)-- (9.5,1.5);
\draw [line width=0.6pt,dotted] (9.5,4)-- (8,2.5);
%
% arrows
\draw [->,line width=0.8pt] (11.6,4.4) -- (9.5,4);
\draw [->,line width=0.8pt] (11.5,1.5) -- (5.734878676629043,2.92);
%
%%%%%%%%%%
%
% LABELS
%
\draw (11.5,1.6) node[text width=3cm,anchor=north west] {partial range of $\Lambda^{r_1}_{\mathscr{R}_{\mathrm{low}}} \text{ for } p=r_1$};
\draw (8,2.9) node[anchor=north west] {$1/2$};
\draw (9.4,1.6) node[anchor=north west] {$1/2$};
\draw (8.8,0.95) node[anchor=north west] {$1/r_1$};
\draw (8.4,0.55) node[anchor=north west] {$1/r_2$};
\draw (7.5,-0.2) node[anchor=north west] {$r=2$};
\draw (2.6,-0.2) node[anchor=north west] {$r_2$};
\draw (-0.4,-0.2) node[anchor=north west] {$r=\infty$};
\draw (4.8,-0.2) node[anchor=north west] {$r_1$};
\draw (11.8,4.8) node[anchor=north west] {$L^2\times L^2 \to L^1$};
\draw (10.1,1.2) node[anchor=north west] {$1/p$};
\draw (11,6) node[anchor=north west] {$1/q$};
\draw (8,0.1) node[anchor=north west] {$0$};
\draw (10.8,2.8) node[anchor=north west] {$1$};
\draw (11,3.4) node[anchor=north west] {$0$};
\draw (11,8.4) node[anchor=north west] {$1$};
\draw (6.7,8.1) node[anchor=south west] {$r$};
%
%%%%%%%%
% 
% POINTS
%
\begin{scriptsize}
\draw [fill=uuuuuu] (1.5,4) circle (2pt);
\draw [fill=sqsqsq] (4.829371517161299,0) circle (2.5pt);
\draw [fill=sqsqsq] (2.62248661252581,0) circle (2.5pt);
\draw [color=uuuuuu] (9.5,4) circle (2pt);
\draw [fill=uuuuuu] (4.12248661252581,4) circle (2pt);
\draw [fill=uuuuuu] (6.329371517161299,4) circle (2pt);
\end{scriptsize}
\end{tikzpicture}
\caption{\footnotesize The interpolation argument in geometric form. The horizontal coordinate represents the exponent $r$, the other coordinates represent the (inverse) exponents $1/p$-$1/q$. A point at coordinates $(r,1/p,1/q)$ represents the $L^p \times L^q \times L^{s'}(\ell^{r'}) \to \R$ estimate for the trilinear form $\Lambda^{r}_{\mathscr{R}}$, where $1/s = 1/p +  1/q$. In particular, the horizontal line through through the axis of the parallelepiped represents (dualized) $L^2\times L^2 \to L^1$ estimates. At $r=\infty$ the range of boundedness for $\Lambda^{\infty}_{\mathscr{R}}$ is the whole $[0,1]\times[0,1]$ (hatched) square in the $1/p$-$1/q$ plane, as given by Theorem \ref{theorem_r_infty}. At exponent $r_1$ intermediate between $2$ and $\infty$, Proposition \ref{proposition_partial_result_r>2_ecc_large} gives for $\Lambda^{r_1}_{\mathscr{R}_{\mathrm{low}}}$ the range of boundedness represented by the thickened interval above, as per label. The other thickened interval represents instead the range of boundedness for $\Lambda^{r_1}_{\mathscr{R}_{\mathrm{high}}}$ given by Proposition \ref{proposition_partial_result_r>2_ecc_small}. Interpolation between exponents $r_1$ and $\infty$ can be interpreted geometrically in the above picture as taking the convex hull of the corresponding endpoint ranges. This gives then for an intermediate exponent $r_2$ a range of boundedness (for $\Lambda^{r_2}_{\mathscr{R}_{\mathrm{low}}}$) represented by the darker shaded area at $r=r_2$, a subset of the $r_2' < p, q < r_2$ range represented by the lighter shaded area. As $r_1$ tends to $2$, the darker shaded area fills the lighter one. Interpolation with the other thickened segment fills the same area in the limit.}
\label{figure_interpolation}
\end{figure}
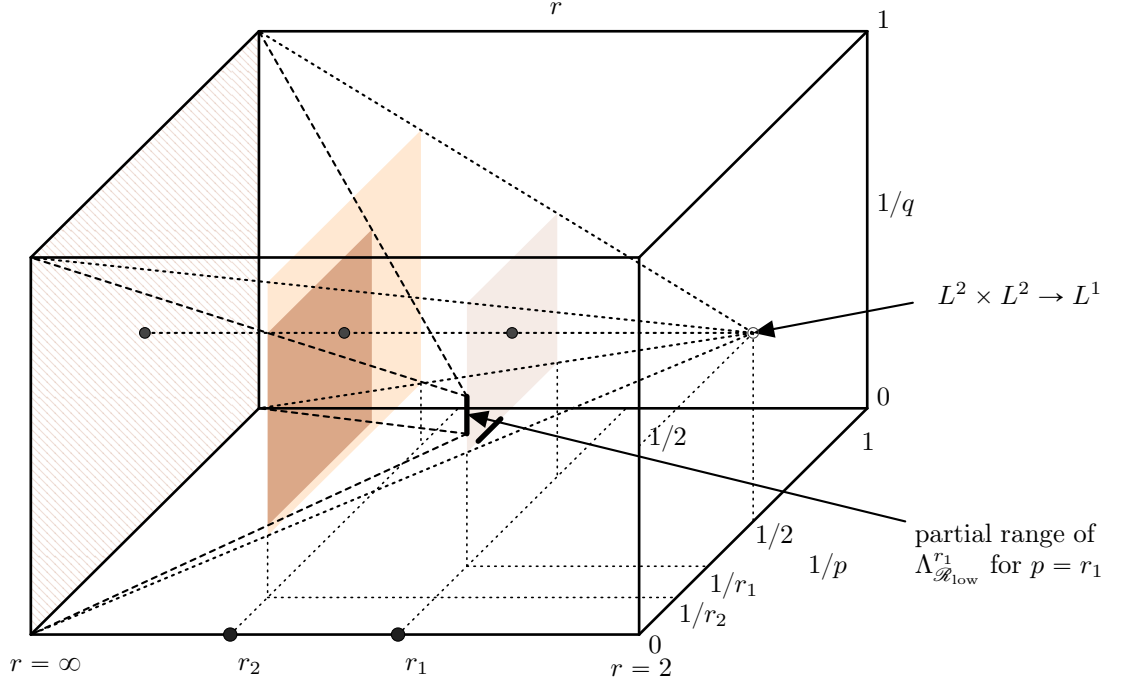
\begin{proof}[Proof of Th. \ref{main_theorem_2}]
First of all, as noticed in Remark \ref{remark:preliminary_result_using_generalized_restricted_weak_type}, the trilinear form $\Lambda^{\infty}_{\mathscr{R}}$ is of generalized restricted weak type $\Big(\Big(\frac{1}{p_0}, \frac{1}{q_0}, \frac{1}{s_0'}\Big); 1 \Big)$ for $1 < p_0,q_0 < \infty$, $1/{p_0} + 1/{q_0} + 1/{s_0'} = 1$ and for any collection $\mathscr{R}$. Secondly, set $r_0 = r_0(r) := (r+2)/2$ in the statement of Proposition \ref{proposition_partial_result_r>2_ecc_large} in order to remove the dependence on $r_0$; write $t_1$ in place of $r'$, so that $r_0(r) = (3t_1 - 2)/(2t_1 - 2)$. Then we have that the trilinear form $\Lambda^r_{\mathscr{R}_{\mathrm{high}}}$ is of generalized restricted weak type $\Big(\Big(\frac{1}{p_1}, \frac{1}{t_1'}, \frac{1}{s_1'}\Big); t_1 \Big)$ for all $1<t_1<2$ and all $p_1, s_1$ such that 
\begin{equation}\label{eqn:exponents_conditions} 
\frac{t_1 - 1}{50(3t_1 - 2)} + \frac{99}{100}\frac{t_1 - 1}{t_1} \leq  \frac{1}{p_1} \leq \frac{2 t_1 - 2}{3 t_1 - 2}  
\end{equation}
and $1/{p_1} + 1/{t_1'} + 1/{s_1'} = 1$. Fix then some other $r>2$. We claim that for any $p,q$ such that $r' < p,q < r$ we can choose $p_0, q_0, p_1, t_1, \theta$ so that 
\begin{equation}\label{eqn:system_for_interpolation}
\begin{aligned}
\frac{1}{p} & = \frac{1 - \theta}{p_0} + \frac{\theta}{p_1}, \\
\frac{1}{q} & = \frac{1 - \theta}{q_0} + \frac{\theta}{{t_1}'}, \\
\frac{1}{r'} & = \frac{1 - \theta}{1} + \frac{\theta}{t_1}, 
\end{aligned}
\end{equation} 
where $1 < p_0, q_0 < \infty$, $t_1$ and $p_1$ satisfy condition \eqref{eqn:exponents_conditions}, and $0 < \theta < 1$. By the above discussion, a direct application of Lemma \ref{lemma_interpolation} will then show that $\Lambda^r_{\mathscr{R}_{\mathrm{high}}}$ is of generalized restricted weak type $\Big(\Big(\frac{1}{p}, \frac{1}{q}, \frac{1}{s'}\Big); r' \Big)$. Notice that the major subset can be taken to be the one given by the value $t_1$, since for $t_0 = 1$ the major subset is the set itself.\\
That we can find such quantities as above follows from the limiting case obtained by taking $t_1 = p_1 = 2$, that is assuming we have the estimate $\big(\big(\frac{1}{2},\frac{1}{2},0\big); 1\big)$ (notice that for such a choice of parameters we don't know if the trilinear form is of generalized restricted weak type, but at the moment we are only concerned with showing the existence of solutions to system \eqref{eqn:system_for_interpolation}). Indeed, in this case it's easy to solve the above system of equations and inequalities: one has $\frac{\theta}{2} = \frac{1}{r}$ and therefore $\frac{1}{p_0} = \Big(\frac{1}{p} - \frac{1}{r}\Big)\Big(1 - \frac{2}{r}\Big)^{-1}$ and $\frac{1}{q_0} = \Big(\frac{1}{q} - \frac{1}{r}\Big)\Big(1 - \frac{2}{r}\Big)^{-1}$. In the general case of $t_1 < 2$ (for which we know the trilinear form is of generalized restricted weak type), by choosing $t_1$ sufficiently close to $2$ we can solve the system by a perturbation of the solution for the limiting case (by continuity), because the range where $\Lambda^r_{\mathscr{R}_{\mathrm{high}}}$ is of generalized restricted weak type according to Proposition \ref{proposition_partial_result_r>2_ecc_large} is arbitrarily close to the $\big(\big(\frac{1}{2},\frac{1}{2},0\big); 1\big)$ estimate.\\
Finally, observe that if $\Lambda^r_{\mathscr{R}_{\mathrm{high}}}$ is of generalized restricted weak type $\Big(\Big(\frac{1}{p}, \frac{1}{q},\frac{1}{s'}\Big); r'\Big)$ then $\widetilde{\Lambda}^r_{\mathscr{R}_{\mathrm{high}}}$ (as given by \eqref{eqn:trilinear_form_1}) is of generalized restricted weak type $\Big(\frac{1}{p}, \frac{1}{q},\frac{1}{s'}\Big)$ \emph{in the classical sense}, and therefore we can conclude the strong $L^p \times L^q \times L^{s'} \to \R$ boundedness of $\widetilde{\Lambda}^r_{\mathscr{R}_{\mathrm{high}}}$ by classical multilinear interpolation. This is equivalent to the $L^p \times L^q \times L^{s'}(\ell^{r'}) \to \R$ boundedness of the trilinear form $\Lambda^r_{\mathscr{R}_{\mathrm{high}}}$, of course.\\
\par
The argument for $\Lambda^r_{\mathscr{R}_{\mathrm{low}}}$ is essentially identical, and one can easily verify that it yields the same range of estimates. Thus we can conclude estimate \eqref{eqn:main_estimate} by summing up the separate contributions given by collections $\mathscr{R}_{\mathrm{high}}$, $\mathscr{R}_{\mathrm{low}}$.
\end{proof} 
\appendix
\section{}\label{appendix_A}
In this appendix we show how to prove Corollary \ref{corollary_smooth_rectangles}. What we want to show in particular is the boundedness of the trilinear form 
\begin{equation}\label{eqn:smooth_trilinear_form}
\Xi^r_{\chi,\mathscr{R}}(f,g,\mathbf{h}):= \sum_{R \in \mathscr{R}} \int\limits_{\R} \iint\limits_{\widehat{\R^2}} \widehat{f}(\xi) \widehat{g}(\eta) \chi_R(\xi, \eta) e^{2\pi i (\xi + \eta) x} h_R(x) \ud \xi \ud \eta \ud x,
\end{equation}
where $\chi_R$ is the smooth compactly supported bump function $\chi$ rescaled to fit the rectangle $R$. The boundedness will also be uniform in $\chi$ if $\chi$ is taken in a bounded subset\footnote{In particular, we ask that $\chi$ has support inside $B(2)$ and that $\|\partial^{\alpha}_{\xi}\partial^{\beta}_{\eta} \chi\|_{L^\infty} \leq C_{\alpha,\beta}$ for some absolute constants $C_{\alpha,\beta}>0$ and all indices $\alpha,\beta$.} of $C_c^{\infty}(\R^2)$. We will perform a series of reductions and then show how the corollary is an immediate consequence of the following variant of Theorem \ref{main_theorem_2}, where dyadic rectangles are replaced by well-separated ones. First let us make the notion of well-separated rectangles rigorous.
\begin{definition}
Let $\mathscr{R}$ be a collection of rectangles. We say that $\mathscr{R}$ is \emph{$K$-separated}, with $K \geq 1$, if for any $R, R' \in \mathscr{R}$ such that $R \neq R'$ we have $KR \cap K R' = \emptyset$.
\end{definition}
With this definition we can state the variant mentioned before.
\begin{theorem*}[Variant of Theorem \ref{main_theorem_2}]%\label{main_theorem_well_separated}
Let $r>2$ and let $\mathscr{R}$ be a $3$-separated collection of \emph{rectangles} in $\widehat{\R^2}$. Let $p,q,s$ be such that 
\[ \frac{1}{p} + \frac{1}{q} = \frac{1}{s} \]
and such that 
\[ r' < p,q < r, \qquad r'/2 < s < r/2. \]
Then for every $f \in L^p$, $g \in L^q$ and $\mathbf{h} \in L^{s'}(\ell^{r'})$ we have
\begin{equation}
\Lambda^r_{\mathscr{R}}(f,g,\mathbf{h}) \lesssim_{r,p,q} \|f\|_{L^p} \|g\|_{L^q} \|\mathbf{h}\|_{L^{s'}(\ell^{r'})}.
\end{equation}
\end{theorem*}
\begin{proof}
We need only observe that if the rectangles are $3$-separated, then for each rectangle $R \in \mathscr{R}$ there exists a rectangle $\widetilde{R}$ with $R \subset \widetilde{R} \subset 3R$ that belongs to one of the dyadic (shifted) grids $\mathcal{D} \times \mathcal{D}'$, where $\mathcal{D},\mathcal{D}'$ are either of the collections\footnote{Notice that $\mathcal{D}_0$ is the standard dyadic grid and $\mathcal{D}_1,\mathcal{D}_2$ also satisfy the combinatorial property $I, I' \in \mathcal{D}_j$ and $I \cap I' \neq \emptyset$ $\Rightarrow$ $I \subseteq I'$ or $I' \subseteq I$.}
\begin{align*}
\mathcal{D}_0 & := \{ [2^k \ell, 2^k(\ell + 1)] \text{ : } k, \ell \in \Z\},  \\
\mathcal{D}_1 & := \{ \big[2^k \big(\ell + \frac{(-1)^k}{3}\big), 2^k\big(\ell + 1 + \frac{(-1)^k}{3}\big)\big] \text{ : } k, \ell \in \Z\},  \\
\mathcal{D}_2 & := \{ \big[2^k \big(\ell - \frac{(-1)^k}{3}\big), 2^k\big(\ell + 1 - \frac{(-1)^k}{3}\big)\big] \text{ : } k, \ell \in \Z\}.
\end{align*}
Then one repeats the proof of Theorem \ref{main_theorem_2} for each such product of grids using rectangles $\widetilde{R}$. The details are left to the reader.
\end{proof}
Now we claim that Corollary \ref{corollary_smooth_rectangles} follows for arbitrary disjoint rectangles if it holds for, say, $12$-separated collections of rectangles. Indeed, let $\mathcal{W}$ denote a Whitney decomposition of $[-1/2,1/2]$ chosen so that $I \in \mathcal{W} \Rightarrow 16 I \subset [-1/2,1/2]$. Thus we have a decomposition of the square $[-1/2,1/2]^2$ given by the rectangles in $\mathcal{W}\times\mathcal{W}$, where elements can be indexed with $(j,k) \in \N^2$ in such a way that the element corresponding to $(j,k)$ has horizontal side of length $\lesssim \delta^j$ and vertical side of length $\lesssim \delta^k$ for some absolute constant $0<\delta < 1$. We consider then a smooth partition of unity $\{\phi_{j,k}\}$ associated to the decomposition $\mathcal{W}\times\mathcal{W}$, that is $\mathds{1}_{[-1/2,1/2]^2} = \sum_{j,k}\phi_{j,k}$, which we can take so that each $\phi_{j,k}(\xi,\eta)$ factorizes as a tensor product $\phi_j \otimes \phi_k$ of functions of one variable and so that each $\phi_{j,k}$ is supported in the $4/3$-enlargement of the rectangle of $\mathcal{W}\times\mathcal{W}$ indexed by $(j,k)$.\\
For a given rectangle $R$, we denote by $\Psi_R$ the affine map that sends $R$ to $[-1/2,1/2]^2$ preserving orientation (thus $\chi_R = \chi \circ \Psi_R$); if the rectangle is the $4/3$-enlargement of the one in $\mathcal{W}\times\mathcal{W}$ indexed by $(j,k)$ we write $\Psi_{jk}$ for this map. We decompose $\chi$ using the partition of unity above: we let $c_{jk}:= \|\chi \phi_{jk}\|_{\infty}$ and $\chi_{jk}:= c_{jk}^{-1} (\chi\cdot \phi_{jk})\circ \Psi_{jk}^{-1}$, so that 
\[ \chi = \sum_{j,k} c_{jk} \chi_{jk} \circ \Psi_{jk}\]
and each $\chi_{jk}$ is $L^\infty$ normalized. It's easy to see that $c_{jk} \lesssim \min(\gamma^j, \gamma^k)$ for some $0<\gamma<1$. Moreover, each $\chi_{jk}$ is adapted to $[-1/2,1/2]^2$ , in the sense that $\chi_{jk}$ is supported in $[-1/2,1/2]^2$ and $\|\partial^{\alpha}_{\xi}\partial^{\beta}_{\eta} \chi_{jk}\|_{L^\infty} \lesssim_{\alpha,\beta} 1$.\\
Using this decomposition of $\chi$ we decompose the trilinear form $\Xi^r_{\chi,\mathscr{R}}$ as follows. Notice that $(\chi\cdot \phi_{jk})\circ \Psi_{R} = c_{jk} \chi_{jk} \circ (\Psi_{jk}\circ\Psi_{R})$, so if we define the rectangle $R_{jk}$ to be $R_{jk} = (\Psi_{jk}\circ\Psi_{R})^{-1}([-1/2,1/2]^2)$ we can see that $\{3/4 R_{jk}\}_{jk}$ is a Whitney decomposition of $R$ which is an affine image of $\mathcal{W}\times\mathcal{W}$. We define the collections $\mathscr{R}_{jk}$ to be $\{R_{jk} \text{ s.t. } R \in \mathscr{R}\}$. Therefore we can write
\[ \Xi^r_{\chi,\mathscr{R}} = \sum_{j,k} c_{jk} \Xi^r_{\chi_{jk},\mathscr{R}_{jk}}, \]
and the boundedness of $\Xi^r_{\chi,\mathscr{R}}$ follows from the (uniform in $j,k$) boundedness of the $\Xi^r_{\chi_{jk},\mathscr{R}_{jk}}$ thanks to the summability of $c_{jk}$. It's easy to verify from the construction above that the rectangles in $\mathscr{R}_{jk}$ are indeed $12$-separated, and hence it suffices to prove boundedness for $12$-separated collections.\\
Now, fix $\chi$ and take $\mathscr{R}$ to be $12$-separated. We can assume that $\chi(\xi,\eta) = \chi_1(\xi)\chi_2(\eta)$ because we can always reduce to this case by a windowed Fourier series expansion\footnote{This has the effect of reducing the separation constant slightly.} (the coefficients will be summable thanks to the fact that $\chi$ is smooth). Letting $\phi_j(x) := \chi_j'(x)$ for $j=1,2$ we can write 
\begin{align*}
\Xi^r_{\chi,\mathscr{R}}& (f,g,\mathbf{h})\\
 & = \sum_{R\in\mathscr{R}} \int\limits_{\R} \iint\limits_{\widehat{\R^2}}\int\limits^{\Psi_{R_2}(\eta)}_{-\infty}\int\limits^{\Psi_{R_1}(\xi)}_{-\infty} \widehat{f}(\xi) \widehat{g}(\eta) \mathds{1}_{R}(\xi,\eta) \phi_1(\zeta)\phi_2(\theta) e^{2\pi i (\xi + \eta)x} h_R(x) \ud\zeta \ud\theta \ud\xi \ud\eta \ud x.
\end{align*}
We want to use Fubini to take the integration in $\ud\zeta\ud\theta$ out, so for $\zeta,\theta$ fixed and $R \in \mathscr{R}$ we let $R_{\zeta,\theta} := R \cap ([\Psi_{R_1}^{-1}(\zeta),+\infty)\times[\Psi_{R_2}^{-1}(\theta),+\infty))$ and let $\mathscr{R}_{\zeta,\theta}:= \{R_{\zeta,\theta} \text{ s.t. } R \in \mathscr{R}\}$. We see that by Fubini we can bound 
\begin{align*}
|\Xi^r_{\chi,\mathscr{R}}& (f,g,\mathbf{h})| \\
& \leq \int\limits^{+\infty}_{-\infty}\int\limits^{+\infty}_{-\infty} |\phi_1(\zeta)||\phi_2(\theta)|\Big|\sum_{R\in\mathscr{R}} \int\limits_{\R} \iint\limits_{\widehat{\R^2}}\widehat{f}(\xi) \widehat{g}(\eta) \mathds{1}_{R_{\zeta,\theta}}(\xi,\eta)  e^{2\pi i (\xi + \eta)x} h_R(x) \ud\xi \ud\eta \ud x\Big| \ud\zeta \ud\theta \\
& = \int\limits^{+\infty}_{-\infty}\int\limits^{+\infty}_{-\infty} |\phi_1(\zeta)||\phi_2(\theta)| |\Lambda^r_{\mathscr{R}_{\zeta,\theta}}(f,g,\mathbf{h})| \ud\zeta \ud\theta.
\end{align*}
To conclude, notice that $\mathscr{R}_{\zeta,\theta}$ is certainly a $3$-separated collection of rectangles and therefore the variant of Theorem \ref{main_theorem_2} above applies, giving a bound independent of $\zeta,\theta$. Since $\phi_1,\phi_2 \in L^1$, integrating in $\ud\zeta\ud\theta$ concludes the argument.

\bibliography{bilinear_rubio_de_francia_bibliography}
\bibliographystyle{amsplain}

\end{document}